\documentclass[USenglish]{article}	

\usepackage[utf8]{inputenc}				
\usepackage[big,online]{dgruyter}	
\usepackage{lmodern} 
\usepackage{microtype}
\usepackage[numbers,square,sort&compress]{natbib}

\usepackage{amsfonts}
\usepackage{amssymb}
\usepackage{array}
\usepackage{graphicx}
\usepackage{color}
\usepackage{stmaryrd}
\usepackage{xifthen}
\usepackage[normalem]{ulem}

\theoremstyle{dgthm}
\newtheorem{theorem}{Theorem}

\newtheorem{proposition}{Proposition}
\newtheorem{lemma}{Lemma}

\theoremstyle{dgdef}

\newtheorem{remark}{Remark}

\newcommand{\T}{\mathcal{T}}
\newcommand{\piola}{\mathcal{P}}

\newcommand{\VT}[1][]{ {\bf VT} ({\ifthenelse{\equal{#1}{}}{\gamma}{#1}}) }
\newcommand{\HT}[1][]{ {\bf HT}^{\ifthenelse{\equal{#1}{}}{1}{#1}}(\gamma) }
\newcommand{\HTS}[1][]{ {\bf HT}^{\ifthenelse{\equal{#1}{}}{1}{#1}}(\mathcal S) }

\newcommand{\F}{\mathcal{F}}
\newcommand{\HH}{{\bf H}}
\newcommand{\A}{{\bf A}}
\newcommand{\B}{{\bf B}}
\newcommand{\I}{{\bf I}}
\newcommand{\ebp}{{\bf e}_{\bp}}
\newcommand{\G}{\mathcal{G}}
\newcommand{\oT}{\overline{T}}

\newcommand{\bP}{{\bf P}}
\newcommand{\bp}{{\bf p}}

\newcommand{\bq}{{\bf q}}
\newcommand{\bQh}{{\bf Q}_h}
\newcommand{\bQ}{{\bf Q}}
\newcommand{\tbp}{\widetilde{\bp}}
\newcommand{\tbq}{\widetilde{\bq}}
\newcommand{\bz}{{\bf z}}
\newcommand{\tbz}{\widetilde{\bz}}
\newcommand{\tbph}{\widetilde{\bp_h}}
\newcommand{\tbpq}{\widetilde{\bq_h}}
\newcommand{\hdivgam}{H({\rm div}_\gamma; \gamma)}
\newcommand{\hdivGam}{H({\rm div}_\Gamma; \Gamma)}

\definecolor{cornellred}{rgb}{0.7, 0.11, 0.11}

\definecolor{darkred}{rgb}{.7,0,0}

\definecolor{green}{rgb}{0,0.7,0}

\definecolor{myblue}{rgb}{0,0,0.7}

\DeclareFontFamily{U}{matha}{\hyphenchar\font45}
\DeclareFontShape{U}{matha}{m}{n}{
      <5> <6> <7> <8> <9> <10> gen * matha
      <10.95> matha10 <12> <14.4> <17.28> <20.74> <24.88> matha12
      }{}
\DeclareSymbolFont{matha}{U}{matha}{m}{n}
\DeclareFontSubstitution{U}{matha}{m}{n}

\DeclareFontFamily{U}{mathx}{\hyphenchar\font45}
\DeclareFontShape{U}{mathx}{m}{n}{
      <5> <6> <7> <8> <9> <10>
      <10.95> <12> <14.4> <17.28> <20.74> <24.88>
      mathx10
      }{}
\DeclareSymbolFont{mathx}{U}{mathx}{m}{n}
\DeclareFontSubstitution{U}{mathx}{m}{n}

\DeclareMathDelimiter{\vvvert}{0}{matha}{"7E}{mathx}{"17}

\begin{document}

	\articletype{Research Article}
	\received{Month	DD, YYYY}
	\revised{Month	DD, YYYY}
  \accepted{Month	DD, YYYY}
  \journalname{De~Gruyter~Journal}
  \journalyear{YYYY}
  \journalvolume{XX}
  \journalissue{X}
  \startpage{1}
  \aop
  \DOI{10.1515/sample-YYYY-XXXX}

\title{A mixed quasi-trace surface finite element method for the Laplace-Beltrami problem}
\runningtitle{Mixed quasi-trace FEM}

\author*[1]{Alan Demlow}
\runningauthor{A. Demlow} 
\affil[1]{\protect\raggedright 
Department of Mathematics, Texas A\&M University, College Station, TX, USA, e-mail: demlow@math.tamu.edu}
	
	
\abstract{Trace finite element methods have become a popular option for solving surface partial differential equations, especially in problems where surface and bulk effects are coupled.  In such methods a surface mesh is formed by approximately intersecting the continuous surface on which the PDE is posed with a three-dimensional (bulk) tetrahedral mesh.  In $H^1$-conforming trace methods, the surface finite element space is obtained by restricting a bulk finite element space to the surface mesh. It is not clear how to carry out a similar procedure in order to obtain other important types of finite element spaces such as $H({\rm div})$-conforming spaces. Following previous work of Olshanskii, Reusken, and Xu on $H^1$-conforming methods, we develop a ``quasi-trace'' mixed method for the Laplace-Beltrami problem.  The finite element mesh is taken to be the intersection of the surface with a regular tetrahedral bulk mesh as previously described, resulting in a surface triangulation that is highly unstructured and anisotropic but satisfies a classical maximum angle condition.  The mixed method is then employed on this mesh.  Optimal error estimates with respect to the bulk mesh size are proved along with superconvergent estimates for the projection of the scalar error and a postprocessed scalar approximation.  {Higher-order surface approximations are also considered.}  
}



\maketitle







\section{Introduction}

Consider the Laplace-Beltrami problem
\begin{align}
\label{lb_problem}
-\Delta_\gamma u=f \hbox{ on } \gamma,
\end{align}
where $\gamma$ is a sufficiently smooth, closed, and compact two-dimensional surface embedded in $\mathbb{R}^3$ and $f$ satisfies the compatibility condition $\int_\gamma f =0$.  Also, $\Delta_\gamma= \nabla_\gamma \cdot \nabla_\gamma$ is the Laplace-Beltrami operator, where $\nabla_\gamma$ is the tangential gradient operator.  

Over the past decades a variety of surface finite element methods have been developed to approximately solve \eqref{lb_problem}.   Some of these methods are more or less direct translations of FEM for Euclidean problems to the surface context, while others are specific to the context of surfaces. In order to construct a surface FEM that is an analogue of a Euclidean FEM, the usual methodology is to first construct a polyhedral surface $\Gamma$ approximating $\gamma$, with $\Gamma$ typically having shape-regular triangular faces.  Standard finite element spaces such as Lagrange \cite{Dz88}, mixed \cite{HS12,FFF16} and standard  \cite{DMS13, ADMSSV15, DM16} and hybridizable \cite{CD16} discontinuous Galerkin elements may be lifted elementwise to $\Gamma$, and implementation is then quite similar to the Euclidean context.  Higher-order polynomial approximations of $\gamma$ may also be used \cite{De09}.

{\it Trace} or {\it cut} surface finite element methods are one of the two main classes of finite element methods specifically designed for approximate solution of surface PDE.  The other main class, {\it narrow band methods}, seems less popular in practice, and we do not discuss it further here as it is not directly relevant to our presentation below; cf. \cite{BDN20} for an overview of parametric, trace, and narrow-band methods.   Trace methods are particularly useful in the context of physical phenomena modeled by systems of partial differential equations posed on domains and surfaces of different dimensions.  Examples of such systems include two-phase flow, which might typically involve a Stokes system posed on a three-dimensional domain along with a surface tension effect modeled by an elliptic surface PDE on the interface between the two phases.  Our goal in this paper is to develop a type of mixed trace surface FEM using standard Raviart-Thomas or Brezzi-Douglas-Marini elements.  

In trace surface FEM an outer (bulk) three-dimensional domain is first triangulated with a tetrahedral mesh $\T_h$.  A surface mesh $\F_h^*$ is  constructed by approximately intersecting $\T_h$ with the surface $\gamma$.  A more precise construction will be discussed as necessary below.  The elements of $\F_h^*$ are {\it arbitrary intersections} of planes with shape-regular tetrahedron.  These elements may thus be either triangles or quadrilaterals, may be of any diameter not exceeding that of the bulk element from which they are derived, and need not satisfy a standard nondegeneracy (shape-regularity or minimum-angle) condition even if the bulk mesh $\T_h$ does.  A trace finite element method is then constructed by first defining a finite element space $B_h$ on the bulk mesh $\T_h$.  In practice this is often a standard Lagrange space.  The surface finite element space $S_h$ is then obtained by simply restricting $B_h$ to the surface mesh $\F_h^*$.  Abstractly the finite element method is then constructed in the usual way using $S_h$.  However, there is not a natural practical basis for $S_h$ since this space is defined in practice using degrees of freedom for the {\it bulk} space $B_h$.  The result is a finite element system that is solved over a spanning set for $S_h$ rather than a basis.  In particular this trace construction artificially introduces degenerate modes into the system of linear equations.  For standard Lagrange trace methods these have proved to not be a major issue, as they can be handled without major difficulty at the linear algebra level \cite{ORG09, OR10} or by stabilization \cite{BHL15, BHLMZ16}.  

Trace finite element methods have been developed for a range of problems including the Laplace-Beltrami problem and more complex scalar elliptic problems  and surface vector Laplace and Stokes problems; cf. \cite{ORG09, BHLMZ20, BDN20, GJOR18, JORZPP, OQRY18} among many others.  In most of these situations the bulk space is relatively simple, and is often related to standard Lagrange spaces or discontinuous analogs.  The reason is that the procedure of restricting a finite element space to a surface can involve a number of subtle and complex issues involving traces and extensions on function  spaces, the interaction between bulk and surface degrees of freedom, and the sometimes delicate considerations needed to ensure that stability holds in the form of a coercivity or inf-sup condition.  These issues are well understood in the context of $H^1$ spaces but are largely unexplored in the context of more complex function spaces such as $H({\rm div})$ and $H({\rm curl})$. A mixed method for the Darcy problems using Masud-Hughes stabilization was given in \cite{HL17, HLM17}, but this method does not use standard mixed finite element spaces or trace analogues of such spaces.  Thus for example tangentiality is enforced only weakly.    As noted in these works,  trace analogues to standard mixed finite element spaces such as Raviart-Thomas and BDM spaces have not been developed, and the extent to which this is possible is not clear.   On the other hand, $H({\rm div})$-conforming spaces have attractive properties such as strongly preserving tangentiality of the vector variable and strong preservation of divergence constraints; cf. \cite{BDL20, LLS20}. 

An essential feature of trace finite element methods is that both the surface $\gamma$ and the bulk surrounding it are triangulated using the {\it same mesh}.  This makes it much simpler in practice to relate physical information from the surface and the bulk.  In contrast, if $\gamma$ and the bulk are meshed separately it is more difficult in practice to pass information between the surface and the bulk.  These issues are compounded in problems where $\gamma$ evolves in time.  Surface triangulations tend to degenerate as $\gamma$ evolves even if the initial mesh is regular, which requires remeshing or other strategies.  In addition, the overhead cost necessary to pass information between the bulk and surface meshes must be born at each time step since the relationship between the two changes constantly.  In trace methods, all of these mesh-related interactions are handled rather naturally even in the context of evolving surfaces.   These desirable features of trace methods are {\it not} dependent on using a surface finite element space derived from a bulk space as described above, only on using a trace  mesh derived from a bulk mesh.  

In \cite{ORX12} it was observed that while the trace surface mesh $\F_h^*$ is in many ways degenerate, when $\T_h$ is shape regular the triangular elements of $\F_h^*$ satisfy a classical maximum-angle condition in that all interior triangle elements are uniformly bounded below $\pi$.  In addition all quadrilateral elements of $\F_h^*$ may be bisected into triangles satisfying the maximum angle condition.  Thus one may easily derive from $\F_h^*$ a conforming triangular mesh $\F_h$ that satisfies a maximum angle condition.  The latter condition is critical because it is known in many contexts to be sufficient to guarantee good approximation properties of the finite element spaces and their natural interpolation operators.  In \cite{ORX12} the authors constructed a surface Lagrange finite element method on the surface mesh $\F_h$.  This ``quasi-trace'' method is mostly identical to the original surface finite element method of Dziuk in \cite{Dz88}, except that the underlying mesh is highly anisotropic and may possess sharp local variations in element sizes instead of being shape-regular.   Optimal-order error estimates may be derived in a straightforward manner using well-known approximation properties of Lagrange spaces on anisotropic meshes \cite{Apel99}.   A disadvantage of this method noted in \cite{ORX12} is that the condition number of the resulting system matrix may be arbitrarily large and the construction of effective preconditioning techniques for iterative solvers is not immediately clear.   This is due to the highly unstructured anisotropic nature of the mesh.  


In this paper we describe and prove error estimates for mixed quasi-trace FEM for the Laplace-Beltrami problem \eqref{lb_problem}.  As noted above, trace methods for \eqref{lb_problem} based on standard conforming Lagrange FEM are well-developed and highly effective, and there appears to be little motivation for using the quasi-trace method developed in \cite{ORX12}.  If however a mixed trace method is desired options are much more limited, and a quasi-trace method is a simple and potentially attractive solution.  Standard mixed finite element methods yield discrete approximations $u_h$ and $\bp_h$ to the scalar solution (e.g., Darcy pressure) $u$ and the solution gradient $\bp=-\nabla_\gamma u$ (e.g., the Darcy flux) of the Laplace-Beltrami problem.  Our methods below are based on well-known Raviart-Thomas (RT) and Brezzi-Douglas-Marini (BDM) finite element spaces {of arbitrary degree.  The original trace surface finite element methods were based on affine surface approximations, but the methodology has subsequently been expanded to allow for higher-degree surface approximations as well \cite{Leh16, GLR18, LR18}.   We accordingly also allow for arbitrary-order trace approximations to the surface under assumptions consistent with those works.  The degrees of the finite element space and surface approximation may be chosen independently in order to balance ``geometric errors'' which arise in the method from approximation of the continuous surface by a discrete one.}  

Our error analysis below includes optimal-order estimates for $\bp-\bp_h$ and $u-u_h$ along with a ``superconvergent'' estimate for $\pi_h u-u_h$, where $\pi_h$ is the $L_2$ projection onto the scalar portion of the finite element space.  In addition we prove error estimates for superconvergent postprocessing techniques that yield a higher-order approximation to the scalar variable $u$.  Such error estimates are all classical in the context of Euclidean mixed methods.  However, even for quasi-uniform surface meshes our estimates do not all appear directly in the literature.  Error estimates for general surface mixed methods appear in the context of finite element exterior calculus in \cite{HS12}.   The derivation and statement of the estimates in that work applies to the methods we consider, but does not include the superconvergence estimates that we do and is not as readily accessible to audiences attuned to classical mixed finite methods.  In \cite{CD16} the authors give a general error analysis for a class of hybridizable discontinuous Galerkin methods that includes the mixed methods considered here.  $L_2$ estimates for the scalar and vector variables are proved along with superconvergence estimates for $\pi_h u-u_h$.  Postprocessing techniques are discussed, but no error estimates are proved.  Finally, \cite{FFF16} contains error analysis of the lowest-order Raviart-Thomas method on polyhedral surfaces in the natural variational $H({\rm div}) \times L_2$ norm.  Such estimates are however insufficient for the $BDM$ elements we consider here due to the difference in approximation order observed in the $L_2$ and $H({\rm div})$ norms for the vector variable.  In summary, our postprocessing error estimates which take geometric errors into account have not previously appeared in the literature even in the context of shape regular meshes.  Our other error estimates largely appear in the literature, but often within a more general proof framework that is not as transparently equivalent to that classically used for mixed methods.  In contrast to previous works we also assume anisotropic rather than shape regular meshes.  For many of our estimates the transition from shape regular to meshes satisfying the maximum angle condition is accomplished by relatively direct application of existing interpolation results on anisotropic meshes, while slightly more care is required to prove error estimates for some postprocessing techniques.  
 
The organization of the paper is as follows.  In Section \ref{sec2} we give preliminaries concerning surface geometry and the mesh and define our finite element method along with associated superconvergent postprocessing techniques.  In Section \ref{sec3} we give some preliminaries needed for error analysis.  The error analysis is contained in Section \ref{sec4}, and in Section \ref{sec5} we give some basic numerical results.

\section{Surface, mesh, and finite element method}
\label{sec2}
\subsection{Surface description}
Let $\gamma \subset \mathbb{R}^3$ be a two-dimensional closed, compact surface embedded in $\mathbb{R}^3$.  We assume that $\gamma$ is at least $C^2$ and that it is furthermore sufficiently smooth as needed below; surface regularity requirements are not a major focus of this work and we do not specify them precisely.  
Under these assumptions, there exists a signed distance function $d:U \rightarrow \mathbb{R}$, where $U \subset\mathbb{R}^3$ is a tubular neighborhood about $\gamma$ whose width is inversely proportional to the maximum principal curvature on $\gamma$.  The function $d$ possesses the same smoothness as $\gamma$, and is negative inside of $\gamma$ and positive outside.  In addition, $\nu(x)=\nabla d(x)$ is the outward pointing unit normal to $\gamma$ for $x \in \gamma$, and we also define the Weingarten map $\HH=D^2 d$. For $x \in U$, there is one zero eigenvalue of $\HH$ corresponding to the eigenvector $\nu$, while the other two eigenvalues $\kappa_1, \kappa_2$ are the principal curvatures of $\gamma$ for $x \in \Gamma$ and are the principal curvatures of parallel surfaces for $x \in U$.  Finally, the closest point projection 
\begin{align}
\bP_d(x)=x-d(x) \nu(x)
\end{align}
is defined on $U$ and (as the name implies) carries $x \in U$ to the unique closest point on $\gamma$.  Note as well that $\nu(x)=\nu(\bP_d(x))$ for $x \in U$.

\subsection{Mesh description}
\label{sec:mesh}
We next discuss trace surface meshes.  Because we will not explicitly use the outer (bulk) mesh in our analysis, we shall discuss the relationship between the bulk and surface mesh somewhat informally, and then develop our method and error estimates under somewhat more abstract assumptions on the surface mesh.  These assumptions have been shown elsewhere to hold for trace meshes.

A typical construction of a surface trace mesh $\F_h^*$ proceeds as follows.   Assume that $\gamma$ is represented by a level set function in that $\gamma = \{x \in \mathbb{R}^3: \psi(x)=0\}$, with $\psi:\mathbb{R}^3 \rightarrow \mathbb{R}$ sufficiently smooth.  We may take $\psi=d$, but this is not necessary.  Let $\Omega \subset \mathbb{R}^3$ be a polyhedral bulk domain containing a sufficiently large neighborhood of $\gamma$, and let $\T_h$ be a quasi-uniform tetrahedral decomposition of $\Omega$.   Finally, let $I_L \psi$ be the $\mathbb{P}_1$ Lagrange interpolant of $\psi$ on $\T_h$ and define the approximate surface {$\Gamma^{lin}=\{x \in \Omega: I_L \psi(x)=0\}$}.  The faces of $\Gamma$ are then arbitrary intersections of planes with tetrahedrons and may be triangular or quadrilaterals; we denote by $\F_h^*$ the set of such faces.  As noted in the introduction and established in \cite{ORX12}, the quadrilateral faces of $\F_h^*$ may be subdivided into triangles in order to yield a conforming triangulation {$\F_h^{lin}$} whose elements all satisfy a maximum angle condition.  

{We additionally consider higher-degree surface approximations.   We assume the existence of a continuous and piecewise sufficiently smooth homeomorphism $\Theta_h: \F_h^{lin} \rightarrow U$ with properties as below.  We denote by $\Gamma=\Theta_h (\Gamma^{lin})$ our final surface approximation and by $\F_h = \{\Theta_h(T):T \in \F_h^{lin}\}$ our final (curved) mesh.  We assume that $\Theta_h$ is the restriction to $\Gamma^{lin}$ of a degree-$k$ vector Lagrange (i.e., Lipschitz continuous) finite element function defined on a subset of the outer mesh $\T_h$, which is generally the case in practice.  The regularity properties of $\Theta_h$ which we describe below are inherited from this ambient mapping.  The polynomial degree $k$ of this mapping then determines the order of convergence of $\Gamma$ to $\gamma$ in various measures.  In the case $k=1$, $\Theta_h$ is the identity.   
}

{Our assumptions on the surface mesh pair ($\F_h^{lin}$, $\F_h$) are as follows.  We note that these can all be shown to be satisfied under reasonable assumptions by meshes constructed as above; cf. \cite{BDN20} and \cite{ORX12} for basic properties of affine trace surface approximations and \cite{Leh16, GLR18, LR18} for construction and properties of higher-order trace surface approximations.  In particular, \eqref{d_assump} and \eqref{discrete_surfs} below may be derived from Corollary 3.2, Lemma 3.7, Lemma 3.8, (3.38), and Lemma 3.14 of \cite{LR18} along with Lemma 3.3 of \cite{GLR18}.  }
\begin{enumerate}
\item The elements of {$\F_h^{lin}$} are triangles satisfying a {\it maximum angle condition}.  That is, the maximum interior angle of triangles lying in {$\F_h^{lin}$} is bounded uniformly below $\pi$, independent of $h$.
\item The mesh {$\F_h^{lin}$} is a decomposition of a closed Lipschitz surface $\Gamma^{lin} \subset U$ having triangular faces, { and the mesh $\F_h$ is likewise a decomposition of a closed Lipschitz surface $\Gamma \subset U$.  }
\item The maximum diameter of elements in $\F_h$ is given by $h$.  In practice $h$ would generally be the maximum diameter of elements lying in a quasi-uniform bulk mesh from which the surface mesh $\F_h$ is derived.
\item The discrete surface $\Gamma$ is transverse to the continuous surface $\gamma$ in the sense that the unit outer normal $\nu_h$ to $\Gamma$ satisfies $\nu_h(x) \cdot \nu(x) \ge c_\nu >0$.
\item The discrete surface $\Gamma$ satisfies
\begin{align}
\label{d_assump}
\|d\|_{L_\infty(\Gamma)} + h\|\nu-\nu_h\|_{L_\infty(\gamma)} \lesssim {h^{k+1}}.
\end{align}
\item { The mapping $\Theta_h$ is a stable bijection between $\Gamma^{lin}$ and $\Gamma$ in the sense that 
\begin{align}
\label{discrete_surfs}
\begin{aligned}
\|Id-\Theta_h\|_{L_\infty(\Gamma^{lin})} +  h(\|I-D \Theta_h \|_{L_\infty(\Gamma^{lin})} + \|\nu_{\Gamma^{lin}}-\nu_h\circ \Theta_h\|_{L_\infty(\Gamma^{lin})}) & \lesssim h^2, 
\\ \|\Theta_h \|_{W_\infty^{k+2}(T)} &  \lesssim 1, ~T \in \F_h^{lin}, 
\\ \|D \Theta_h\|_{\T_h} + \|D \Theta_h^{-1} \|_{L_\infty(\Theta_h(\T_h))} & \simeq 1, 
\\ \int_{\Gamma^{lin}} u(\Theta(x)) \mu_h(x)  = \int_\Gamma u \hbox{ for all } u \in L_1(\Gamma),  \hbox{ with } \|1-\mu_h\|_{L_\infty(\Gamma^{lin})} & \lesssim h.
\end{aligned} 
\end{align}
}
\end{enumerate}

\subsection{Differential operators, function spaces, and the mixed Laplace-Beltrami problem}  Given a scalar function $u:\gamma \rightarrow \mathbb{R}$, the tangential gradient can be defined by first extending $u$ to $U$ by $u(x)=u(\bP_d(x))$, $x \in U$.  The tangential gradient is given by $\nabla_\gamma u(x)=  \Pi \nabla u(x)$, where $\nabla$ is the Euclidean gradient and $\Pi = {\bf I}-\nu \otimes \nu$ is the projection onto the tangent plane of $\gamma$.  {We also denote by $\Pi_\Gamma= {\bf I}-\nu_h \otimes \nu_h$ the projection onto the tangent plane of $\Gamma$.}  In addition we define ${\rm div}_\gamma=\nabla_\gamma \cdot$ and $\Delta_\gamma = \nabla_\gamma \cdot \nabla_\gamma$.  These definitions are intrinsic, that is, they only depend on the value of $u$ on $\gamma$ even though we use an extension of $u$ to $\mathbb{R}^3$.  Similar definitions for $\nabla_\Gamma$ and ${\rm div}_\Gamma$ hold.  Other tangential derivatives of functions on $\gamma$ may be taken by iteratively applying the above-defined operators.

We next define the function spaces $H^1(\gamma)=\{u \in L^2(\gamma): \nabla_\gamma u \in [L_2(\gamma)]^3\}$ and $H({\rm div}_\gamma, \gamma) = \{\vec{u} \in [L_2(\gamma)]^3 : \vec{u} \cdot \nu= 0 \hbox{ and } {\rm div}_\gamma \vec{u} \in L_2(\gamma)\}$.  Note that $H({\rm div}_\gamma; \gamma)$ consists of {\it tangential} $L_2$ vector fields with $L_2$ tangential divergence.   Other function spaces such as $H^2(\gamma)$ may be similarly defined.  Finally, we let $L_2^0(\gamma)= \{f \in L_2(\gamma): \int_\gamma f = 0\}$.  

We now formulate the standard mixed weak form of the Laplace-Beltrami problem.  Let first
\begin{align}
\label{gamma_forms}
\begin{aligned}
a(\bp, \bq)&=\int_\gamma \bp \cdot \bq, ~\bp, \bq \in [L_2(\gamma)]^3,
\\ b(\bp, u) & = \int_\gamma u {\rm div}_\gamma \bp, ~u \in L_2(\gamma) \hbox{ and } \bp \in H({\rm div}_\gamma, \gamma),
\\ (u,v)&= \int_\gamma uv, ~~u,v \in L_2(\gamma).
\end{aligned}
\end{align}
Given $f \in L_2^0(\gamma)$, we seek $(\bp, u ) \in H({\rm div}_\gamma; \gamma) \times L_2^0(\gamma)$ such that
\begin{align}
\label{mixed_prob}
\begin{aligned}
a(\bp, \bq) - b(\bq, u)& =0, ~~\bq \in H({\rm div}_\gamma; \gamma), 
\\ b(\bp, v) & = (f, v), ~~v \in L_2(\gamma).
\end{aligned}
\end{align}
This problem is equivalent to the standard weak form of the Laplace-Beltrami problem \eqref{lb_problem} under the given assumptions.  

\subsection{Finite element method and postprocessing}

We shall use the classical Brezzi-Douglas-Marini and Raviart-Thomas spaces for approximating $H({\rm div}) \times L_2$.  Let $\hat{T} \subset \mathbb{R}^2$ be the unit reference triangle.  In order to define $BDM$ spaces we take $\bQh (\hat{T}) = [\mathbb{P}_r]^2$ and $V_h(\hat{T})= \mathbb{P}_{r-1}(\hat{T})$ for $k \ge 1$.  We then set $BDM_r(\hat{T}) = \bQh (\hat{T}) \times V_h(\hat{T})$.   For the Raviart-Thomas family of spaces we take $\bQh(\hat{T}) = [\mathbb{P}_r]^2 \oplus {\bf x} \mathbb{P}_r$ and $V_h (\hat{T})= \mathbb{P}_r(\hat{T})$ for $r \ge 0$, and set $RT_r = \bQh(\hat{T}) \times V_h(\hat{T})$.  

In order to define corresponding spaces on $\Gamma$, we fix 
{$\overline{T} \in \F_h^{lin}$ and let $B:\hat{T} \rightarrow \overline{T}$ be one of the possible natural affine transformations mapping the unit reference element $\hat{T}$ to $\overline{T}$, with 3$\times$2 constant derivative ${\bf B}$.  Let then $A=\Theta_h \circ B: \hat{T} \rightarrow T \in \F_h$ with 3$\times$2 matrix derivative $\A$ and Jacobian $|\A|=\sqrt{\A^{\top} \A}$.  }
Then we let
\begin{align}
\bQh= \{\bp_h \in H({\rm div}_{\Gamma}; \Gamma): \forall ~T \in \F_h, \bp_h|_{T} = \frac{1}{|\A|} \A \hat{\bp}_T \circ A^{-1}, \hbox{ some } \hat{\bp}_T \in \bQh(\hat{T}) \}.
\end{align}
Members of ${\bf Q}_h$ are elementwise Piola transforms of finite element vector functions defined on the reference element, with the additional restriction that conormal components must be continuous across element boundaries.  Note that because {the normal vector $\nu_h$ is not continuous across element boundaries}, the conormals coming from two triangles $T_1, T_2 \in \F_h$ sharing a given edge may not be colinear.  We also define
\begin{align}
\begin{aligned}
V_h & = \{v_h \in L_2(\Gamma): \forall T \in \F_h, v_h|_T = \hat{v}_T \circ A^{-1}, \hbox{ some } \hat{v}_T \in V_h(\hat{T})\},
\\ V_h^0 & = V_h \cap L_2^0(\Gamma).
\end{aligned}
\end{align}
Finally we set {$RT_r$ or $BDM_r$} equal to ${\bf Q}_h \times V_h$, depending on the choice of the corresponding reference spaces.  {We also denote by $\bQ_h^{lin} \times V_h^{lin}$ the corresponding spaces defined on $\Gamma^{lin}$ by Piola transform or mapping via $B$ as appropriate.    }

We next define the forms $a_\Gamma$, $b_\Gamma$, and $(\cdot, \cdot)_\Gamma$ to be the obvious counterparts on $\Gamma$ of the forms defined on $\gamma$ in \eqref{gamma_forms}:
\begin{align}
\label{Gamma_forms}
\begin{aligned}
a_\Gamma(\bp, \bq)&=\int_\Gamma \bp \cdot \bq, ~\bp, \bq \in [L_2(\Gamma)]^3,
\\ b_\Gamma (\bp, u) & = \int_\Gamma u {\rm div}_\Gamma \bp, ~u \in L_2(\Gamma) \hbox{ and } \bp \in H({\rm div}_\Gamma, \Gamma),
\\ (u,v)_\Gamma&= \int_\Gamma uv, ~~u,v \in L_2(\Gamma).
\end{aligned}
\end{align}
Our finite element method is as follows.  Given $f_h \in L_2^0(\Gamma)$, we seek $(\bp_h, u_h ) \in \bQh \times V_h$ such that 
\begin{align}
\label{mixed_fem}
\begin{aligned}
a_\Gamma (\bp_h, \bq_h) - b_\Gamma (\bq_h, u_h) & =0, ~~\bq_h \in \bQh, 
\\ b_\Gamma (\bp_h, v_h) & = (f_h, v_h), ~~v_h \in V_h,
\\ (u_h, 1)_{\Gamma} & =0.
\end{aligned}
\end{align}

An important feature of standard mixed and also hybridizable discontinuous Galerkin methods is the ability to define a postprocessed solution to the pressure (scalar) variable that converges at a rate one order higher than the original scalar approximation; cf. \cite{Sten91}.  Such estimates rely on the superconvergence of the finite element solution to the $L_2$ projection of the continuous solution proved in the preceding subsections.  Superconvergent postprocessing techniques for surface mixed and HDG FEM are discussed in \cite{CD16} and numerical results given, but no proofs were given.

We define two well-known postprocessing techniques.  {Below we employ either $RT_{r-1}$ or $BDM_r$ elements in our finite element method and recover a postprocessed approximation to $u$ in locally $\mathbb{P}_r$ spaces.  Thus our postprocessed solution consists of local polynomials of one degree higher than the scalar space in the given mixed method.}   Here and below we fix a triangle $T \in \T_h$.   {Let $\mathbb{P}_r^0 = \{v \in \mathbb{P}_r(T): \int_T v = 0\}$.  For our first technique we seek $u_h^* \in \mathbb{P}_r(T)$} such that
\begin{align}
\label{post_def}
\begin{aligned}
\int_T \nabla_\Gamma u_h^* \nabla_\Gamma v_h & = \int_T f_h^* v_h - \int_{\partial T} \bp_h\cdot {\bf n}_T v_h, ~v_h \in {\mathbb{P}_r^0},
\\  \int_T u_h^* = \int_T u_h.
\end{aligned}
\end{align}
Here ${\bf n}_T$ is the outward-pointing unit conormal on $\partial T$, and $f_h^*$ is an approximation to $f^\ell$ that is not required to be the same as $f_h$.  Alternately we may replace the first equation above by
\begin{align}
\label{post_def2}
\int_T \nabla_\Gamma u_h^* \nabla_\Gamma v_h & = -\int_T \bp_h \cdot \nabla_\Gamma v_h, ~~v_h \in { \mathbb{P}_r^0.}
\end{align}

{Below we give error estimates for these methods which incorporate the effects of geometric errors; to our knowledge these are the first such estimates in the literature even for shape regular meshes.}

\section{Error analysis preliminaries}
\label{sec3}

{In this section we collect results on equivalence of norms and comparison of forms over different surfaces, and also on interpolation operators.  }

\subsection{Lifts, extensions, and Piola transforms}
\label{subsec:lifts}
In this section we discuss relationships between functions defined on the continuous surface $\gamma$ and those defined on the discrete surface $\Gamma$. Given a scalar function $u:\gamma \rightarrow \mathbb{R}$, its extension to $u$ is given by $u^\ell(x) = u(\bP_d(x))$; the restriction of $u^\ell$ to $\Gamma$ then defines an appropriate function on $\Gamma$.  Given $u_h$ defined on $\Gamma$ and $x \in \Gamma$, let $u_h^\ell(\bP_d(x))=u_h(x)$.  We may extend $u_h^\ell$ to all of $U$ by taking $u_h^\ell$ to be constant in the direction normal to $\gamma$.  

Next let $\mu$ satisfy
\begin{align}
\label{mu_def}
\int_\Gamma f(x) \mu = \int_\gamma f^\ell(x), \hbox{ all } f \in L_1(\Gamma).
\end{align}
We then have that {(cf. \cite{De09, DD07})}
\begin{align}
\mu(x)=\nu \cdot \nu_h (1-d(x) \kappa_1(x)) (1-d(x) \kappa_2(x)),
\end{align}
and
\begin{align}
\label{mu_bound}
|1-\mu| \lesssim {h^{k+1}}, ~~\mu \simeq 1.
\end{align}

We next define the Piola transformation of $H({\rm div})$ tangential vector functions between the surfaces $\gamma$ and $\Gamma$ (cf. \cite{CD16}).  Given $\bp_h \in H({\rm div}_\Gamma; \Gamma)$, let 
\begin{align}
\label{piola_def}
(\piola_{\bP_d} \bp_h)(\bP_d(x)) = \frac{1}{\mu} [\Pi-d \HH] \bp_h(x), ~~x \in \Gamma.
\end{align}
Given $\bp \in H({\rm div}_\gamma; \gamma)$, we also define
\begin{align}
\label{inv_piola_def}
\piola_{\bP_d^{-1}}\bp(x) = \mu \left [{\bf I}-\frac{\nu \otimes \nu_h}{\nu \cdot \nu_h} \right ] [ {\bf I}-d \HH]^{-1} \bp(\bP_d(x)), ~~x \in \Gamma.
\end{align}
An elementary computation yields $\piola_{\bP_d} \piola_{\bP_d^{-1}} \bp = \bp$.  In addition, $\bp \in H({\rm div}_\gamma; \gamma)$ if and only if $\piola_{\bP_d^{-1}} \bp \in H({\rm div}_{\Gamma}; \Gamma)$.   {We also give a concrete definition of the Piola transform $\piola_{\Theta_h}$.  As above, fix $\overline{T} \in \F_h^{lin}$ and let $B:\hat{T} \rightarrow \overline{T}$ be one of the possible natural affine transformations mapping the unit reference element $\hat{T}$ to $\overline{T}$, with 3$\times$2 constant derivative ${\bf B}$.  Let then $A=\Theta_h \circ B: \hat{T} \rightarrow T \in \F_h$ with 3$\times$2 matrix derivative $\A$ and Jacobian $|\A|=\sqrt{\A^{\top} \A}$.   Let ${\bf B}^{+}=(\B^{\top} \B)^{-1} \B^{\top}$ be the Moore-Penrose pseudoinverse of $\B$.  Then $\piola_{\Theta_h} \overline{\bp}_h= \frac{|\B|}{|\A|} \A \B^{+} \overline{\bp}_h$, and $\piola_{\Theta_h^{-1}} \bp_h = \frac{|\A|}{|\B|} \B \A^{+} \bp_h$.  Note also that $\mu_h = \frac{|\A|}{|\B|}$; cf. \eqref{discrete_surfs}.}

We additionally have that 
\begin{align}
\label{piola_div}
\begin{aligned}
{\rm div}_{\Gamma} \bp_h &= \mu {\rm div}_\gamma \piola_{\bP_d} \bp_h, ~~~ {\rm div}_\gamma \bp = \frac{1}{\mu} {\rm div}_\Gamma \piola_{\bP_d^{-1}} \bp,
\\ { {\rm div}_{\Gamma^{lin}} \overline{\bp}_h}  & = {\mu_h {\rm div}_\Gamma \piola_{\Theta_h} \overline{\bp}_h, ~~~ {\rm div}_\Gamma \bp_h = \frac{1}{\mu_h} {\rm div}_{\Gamma^{lin}} \piola_{\Theta_h^{-1}} \bp_h}
\end{aligned}
\end{align}
for $\bp \in H({\rm div}_\gamma; \gamma)$, $\bp_h \in H({\rm div}_\Gamma; \Gamma)$, and {$\overline{\bp}_h \in H({\rm div}_{\Gamma^{lin}}; \Gamma^{lin})$}.  Note that from \cite{CD16} we have for $v \in H^1(\gamma)$ and $\bq=-\nabla_\gamma v$ that
\begin{align}
\label{piola_dif}
|\nabla_\Gamma v^\ell + \piola_{\bP_d^{-1}}\bq| \lesssim {h^{k+1}} |\bq^\ell|.
\end{align}
The proof of  this inequality and also \eqref{mu_bound} above rely only on the assumption \eqref{d_assump} and not on shape regularity.

\subsection{Equivalence of norms and forms}
The material in this section is largely found in \cite{DD07, CD16}.  Below and in what follows we shall use a superscript ``tilde'' to denote a Piola transform.  That is, given $\bp \in H({\rm div}_\gamma; \gamma)$ we let $\tbp= \piola_{\bP_d^{-1}} \bp$, and given $\bp_h \in H({\rm div}_\Gamma; \Gamma)$ we let $\tbph = \piola_{\bP_d} \bp_h$.  From \eqref{piola_div} we then have
\begin{align}
\label{b_equiv}
b(\bp, u)=b_\Gamma(\tbp, u^\ell){=b_{\Gamma^{lin}}( \piola_{\Theta_h^{-1}} \tbp, u^\ell \circ \Theta_h)  }, ~~\bp \in H({\rm div}_\gamma; \gamma) \hbox{ and } u \in L_2(\gamma).
\end{align}
This identity also holds over any subset $T$ of $\Gamma$ and its image $\bP_d(T) \subset \gamma$.

In contrast, the quantities $a(\bp, \bq)$ and $a_\Gamma (\tbp, \tbq)$ are not equal and will be a main source of ``geometric errors'' in the discretization of \eqref{mixed_prob} over the discrete surface $\Gamma$.  We quantify this difference as follows.  Similar results are found in \cite{FFF16}, but using somewhat different notation, so we provide a brief proof.
\begin{proposition}
Given $\bp, \bq \in \hdivgam$, there holds
\begin{align}
\label{a_equiv} 
a(\bp, \bq) - a_\Gamma(\tbp,\tbq) = a([\Pi-\B_h] \bp, \bq),
\end{align}
where $\B_h = \mu \Pi[\I-d\HH]^{-1} \left [ \I - \frac{\nu_h \otimes \nu}{\nu \cdot \nu_h} \right ] \left [ \I - \frac{\nu \otimes \nu_h}{\nu \cdot \nu_h} \right ] [\I -d\HH]^{-1} \Pi$.
In addition, $|\Pi-\B_h| \lesssim {h^{k+1}}$, so that
\begin{align}
\label{a_pert_bound}
|a(\bp,\bq)-a_\Gamma(\tbp, \tbq)| \lesssim {h^{k+1}} \|\bp\|_\gamma \|\bq\|_\gamma.
\end{align}
\end{proposition}
\begin{proof}
The relation \eqref{a_equiv} follows from the definition \eqref{inv_piola_def} and $\Pi \bp = \bp$.  In addition,  $\Pi [\I-d\HH]^{-1} = [\I-d\HH]^{-1}$,  and $|1-\nu\cdot \nu_h| =\frac{1}{2} |(\nu-\nu_h) \cdot (\nu-\nu_h)| \lesssim {h^{2k}}$ by \eqref{d_assump}.  Also, it is not hard to check that $|\I - [\I-d\HH]^{-1}| \lesssim {h^{k+1}}$ for $h$ small enough.  Employing also \eqref{mu_bound}, further rearranging terms, and recalling that $|\Pi \Pi_\Gamma\Pi -\Pi| = |(\nu_h- \nu \cdot \nu_h \nu) \otimes (\nu_h - \nu \cdot \nu_h \nu)| \lesssim {h^{2k}}$ then yield that $\Pi-\B_h = \Pi-\Pi \Pi_\Gamma\Pi + O({h^{k+1}})$ and so
\begin{align}
\label{B_bound}
|\Pi -\B_h| \lesssim {h^{k+1} + |\Pi-\Pi \Pi_\Gamma \Pi| \lesssim h^{k+1}}.
\end{align}
\end{proof}
We shall also need the following norm equivalences.
\begin{align}
\label{norm_equiv}
\begin{aligned}
{\|u^\ell \circ \Theta\|_{L_2(\Gamma^{lin})} \simeq} \|u^\ell\|_\Gamma & \simeq \|u\|_\gamma, ~~u \in L_2(\gamma),
\\ { \|\piola_{\Theta_h^{-1}} \tbp \|_{\Gamma^{lin}} \simeq } \|\tbp\|_\Gamma & \simeq \|\bp\|_\gamma, ~~\bp \in \hdivgam,
\\ { \|\piola_{\Theta_h^{-1}} \tbp \|_{H_h^1 (\Gamma^{lin})}  \simeq }\|\tbp\|_{[H_h^1(\Gamma)]^3} & \simeq \|\bp\|_{[H^1(\gamma)]^3}, ~~ \tbp \in \hdivgam \cap [H^1(\gamma)]^3.
\\ { \|u^\ell \circ \Theta_h\|_{H_h^{r+1}(\Gamma^{lin})} \lesssim}  \|u^\ell\|_{H_h^{r+1} (\Gamma)} &  \lesssim \|u\|_{H^{r+1} (\gamma)}, ~u \in H^{r+1} (\gamma).
\end{aligned}
\end{align}
Here $H_h^1$ denotes a broken (elementwise) $H^1$ norm and similarly for $H_h^{r+1}$.  {Also, a norm with a domain subscript will denote an $L_2$ norm over that domain, e.g., $\|\cdot\|_{\Gamma}= \|\cdot \|_{L_2(\Gamma)}$.}  Note also that the next-to-last inequality is proved in \cite{BDL20} {for linear surface approximations.  More generally these inequalities follow from the previously made assumptions.}  

\subsection{Interpolation operators}
\label{subsec:interp}

We shall also need interpolation error estimates for the above spaces on anisotropic meshes.   Stability and error estimates for canonical interpolation operators on {elementwise flat} anisotropic meshes can be found in \cite{AK20} for $BDM$ elements and in \cite{AADL11,AD99,DL08} for $RT$ elements.  The estimates given in these papers are finer than what we state and require below as they encode directional information about the triangles instead of merely using the element diameter in the error estimate.  In trace methods the anisotropy of the mesh is an artifact of the mesh generation technique and is not related to anisotropic behavior of the PDE solution.  In addition, the anisotropy of our meshes is highly unstructured, with elements sharing a single vertex potentially having vastly different aspect ratios, diameters, and orientations, as illustrated in Figure \ref{fig2}.  Finally, it is shown in \cite{DO12} that in every trace element patch there is a surface element that is shape regular and has diameter equivalent to the bulk element from which it is derived.  Thus error estimates that sharply take into account elementwise size and directional information will not generally capture meaningful information.  We instead derive error estimates with respect to the maximum mesh diameter $h$.   {We also incorporate estimates on the curved surface $\Gamma$ for $k\ge 2$; these may be derived from the corresponding estimates for flat meshes using the assumptions \eqref{discrete_surfs} and the definitions of $\piola_{\Theta_h}$ and $\piola_{\Theta_h^{-1}}$.}

\setlength{\unitlength}{.75cm}
\begin{figure}[h]
\centering
\includegraphics[scale=1]{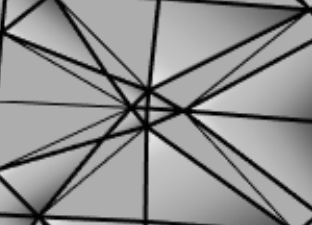}
\caption{Illustration of the high local variation of element diameters, aspect ratios, and orientations in a typical trace mesh.}
\label{fig2}
\end{figure}

{
\begin{lemma}
\label{lemma:lin_surf_interp}
Let $\tilde{\bQ}^{lin} = \{\bp \in H({\rm div}_\Gamma^{lin}; \Gamma^{lin}) : \bp_T \in [H_1(\overline{T})]^3, \hbox{ all } \overline{T} \in \F_h^{lin}\}$, and similarly for $\tilde{\bQ}$ on $\Gamma$.  There exists a projection $I_h^{lin}:\tilde{\bQ}^{lin} \rightarrow \bQh^{lin}$ such that the {\rm commuting diagram property} holds:  
\begin{align}
\label{cdp}
{\rm div}_{\Gamma^{lin}}  I_h^{lin} \bp = \pi_h^{lin} {\rm div}_{\Gamma^{lin}} \bp, ~~ \bp \in \tilde{\bQ}^{lin},
\end{align}
where $\pi_h^{lin}:L_2(\Gamma^{lin}) \rightarrow V_h^{lin}$ is the $L_2$ projection.  We also extend the definitions of the above interpolants to $T \in \F_h$:
\begin{align}
\label{interp_defs}
\begin{aligned}
I_h \bp & = \piola_{\Theta_h} I_h^{lin} \piola_{\Theta_h^{-1}} \bp, ~\bp \in \tilde{\bQ},
\\\pi_h v & = (\pi_h^{lin} v \circ \Theta_h) \circ \Theta_h^{-1}.
\end{aligned}
\end{align}
Let $\overline{T} \in \F_h^{lin}$ and $T =\Theta_h(\overline{T})$.  Given $\bp, u$ defined on $T$ with $\overline{\bp} = \piola_{\Theta_h^{-1}} \bp$ and $\overline{u}=u \circ \Theta_h$ possessing the requisite regularity, there then hold the following interpolation error bounds:  
\begin{align}
\label{interp_bounds}
\begin{aligned}
\|\bp-I_h \bp\|_{L_2({T})}+ \|\overline{\bp}-I_h^{lin} \overline{\bp}\|_{L_2(\overline{T})} & \lesssim h_T^{\ell} \|\bp\|_{H^{\ell}(T)} \lesssim h_T^{\ell} \|u\|_{H^{\ell+1}(T)},
\\ \|u-\pi_h u\|_{L_2(T)} + \|\overline{u} - \pi_h^{lin}\overline{u}\|_{L_2(\overline{T})} &  \lesssim h_T^{m} \|u\|_{H^m(T)}, 
\end{aligned}
\end{align}
where $1 \le \ell \le r+1$, and $0 \le m \le r$ for $BDM_r$ elements and $0 \le m \le r+1$ for $RT_r$ elements.  In addition we have for $BDM$ elements that
\begin{align}
\label{BDM_stable}
\|I_h \bp\|_{L_2(T)} + \|I_h^{lin} \overline{\bp} \|_{L_2(\overline{T})} \lesssim \|\bp\|_{L_2(T)} + h_T |\bp|_{H^1(T)}
\end{align}
and for Raviart-Thomas elements that
\begin{align}
\label{RT_stable}
\|I_h \bp\|_{L_2(T)}+ \|I_h^{lin} \overline{\bp}\|_{L_2(\overline{T})} \lesssim \|\bp\|_{L_2(T)} + h_T |\bp|_{H^1(\overline{T})} + \|{\rm div}_\Gamma \bp\|_{L_2(T)}.
\end{align}
\end{lemma}
}

{We shall also use estimates for the Lagrange interpolant on meshes satisfying the maximum angle condition; cf. \cite{Apel99}.
Let $I_L^{lin}$ be the Lagrange interpolant into the conforming elementwise $\mathbb{P}^\ell$ functions over $\Gamma^{lin}$, and $I_L v  = (I_L^{lin} v \circ \Theta_h) \circ \Theta_h^{-1}$.  Then for $\oT \in \F_h^{lin}$ and $T=\Theta_h(\oT) \in \F_h$, 
\begin{align}
\label{lagrange_est}
\|u-I_L u\|_{L_2(T)} +  \|\overline{u}-I_L^{lin} \overline{u}\|_{L_2(\oT)} + h_T\|\nabla_\Gamma (u-I_L^{lin} u)\|_{L_2(\oT)} \lesssim h_T^m |u|_{H^m(T)}, ~2 \le m \le \ell+1.  
\end{align}
}

{ Note carefully that $\pi_h$ is {\it not} the $L_2$ projection onto $V_h$ over $\Gamma$, but rather the $L_2$ projection onto $V_h^{lin}$ of functions lifted to $\Gamma^{lin}$.  In addition, the commuting diagram property \eqref{cdp} does not hold directly for $I_h$ if $k\ge 2$.  In fact, ${\rm div}_{\Gamma^{lin}} \bQ_h^{lin} \subset V_h^{lin}$, but for $k \ge 2$ ${\rm div}_\Gamma \bQ_h \not\subset V_h$. Rather, the following {\it weak commuting diagram property} holds.  
\begin{proposition}[Weak commuting diagram property]
Given $\bp \in \tilde{\bQ}$ and $v \in L_2(\Gamma)$, there holds
\begin{align}
\label{wcdp}
b_\Gamma (I_h \bp, v) = b_{\Gamma}(\bp, \pi_h v), ~~v \in L_2(\Gamma).
\end{align}
\end{proposition}
\begin{proof}
%
Using \eqref{b_equiv}, \eqref{cdp}, and \eqref{interp_defs} yields
\begin{align*}
b_\Gamma(I_h \bp, v) = b_{\Gamma^{lin}} (I_h^{lin} \piola_{\Theta_h^{-1}} \bp, v \circ \Theta_h)=b_{\Gamma^{lin}} (\piola_{\Theta_h^{-1}} \bp, \pi_h^{lin} v\circ \Theta_h)=b_\Gamma(\bp, \pi_h v).  
\end{align*}
\end{proof}
}

\section{Error analysis}
\label{sec4} 
\subsection{Inf-sup condition}
\begin{lemma} There is a constant $\beta>0$ independent of $h$ such that for all $v_h \in V_h^0$, 
\begin{align}
\label{discrete_inf_sup}
\sup_{\bq_h \in \bQh} \frac{b_\Gamma (\bq_h, v_h)}{\|\bq_h\|_{H({\rm div}_\Gamma; \Gamma)}} \ge \beta \|v_h\|_{L_2(\Gamma)}.
\end{align}
\end{lemma}
\begin{proof}
Given $v_h \in V_h^0$, let $z \in H^2(\gamma)$ solve $-\Delta_\gamma z= \frac{1}{\mu} v_h^\ell$.  Let also $\bq = \nabla_\gamma z$ so that ${\rm div}_\gamma \bq = \frac{1}{\mu} v_h^\ell$.  By $H^2$ regularity we also have $\|\bq\|_{H^1(\gamma)} \lesssim \|\frac{1}{\mu} v_h^\ell\|_{L_2(\gamma)}$.  Let $\tilde{\bq} = \piola_{\bP^{-1}} \bq$.  From \eqref{piola_div} we have that ${\rm div}_\Gamma \tilde{\bq} = v_h$, and from the {weak} commuting diagram property {\eqref{wcdp}} we thus deduce that {$b_\Gamma (I_h \tilde{\bq}, w_h) = (v_h, w_h)$, $w_h \in V_h$.}  Using the $H^1$ stability of $I_h$ \eqref{BDM_stable} or \eqref{RT_stable} as appropriate, norm equivalence \eqref{norm_equiv}, $H^2$ regularity, and finally \eqref{mu_bound} and \eqref{norm_equiv} we thus have that
\begin{align}
\label{eq400}
\begin{aligned}
\sup_{\bq_h \in \bQh} \frac{b_\Gamma (\bq_h, v_h)}{\|\bq_h\|_{H({\rm div}_\Gamma; \Gamma)}} &  \ge \frac{b_\Gamma ( I_h \tilde{\bq}, v_h)}{\|I_h \tilde{\bq}\|_{H({\rm div}_\Gamma; \Gamma)} } 
\gtrsim \frac{\|v_h\|_\Gamma^2}{\|\tbq\|_{[H_h^1(\Gamma)]^3}} 
\\ & \gtrsim \frac{\|v_h\|_\Gamma^2}{\|\bq\|_{[H^1(\gamma)]^3}}
 \gtrsim \frac{\|v_h\|_{\Gamma}^2}{ \|\frac{1}{\mu}v_h^\ell\|_{\gamma}}
 \gtrsim \|v_h\|_\Gamma.
\end{aligned}
\end{align}
This completes the proof. 
\end{proof}

\subsection{Basic error estimates}
We first establish some relationships that will help us to encode geometric errors due to the approximation of $\gamma$ by $\Gamma$ when proving error bounds.  Using \eqref{b_equiv},\eqref{a_equiv}, along with \eqref{mixed_prob}, we first compute that for $\bp$ satisfying \eqref{mixed_prob} and $\bq \in \hdivgam$, 
\begin{align}
\label{a_pert}
a_\Gamma(\tbp, \tbq) - b_\Gamma(\tbq, u^\ell) = a_\Gamma(\tbp, \tbq)-a(\bp, \bq) \lesssim {h^{k+1}} \|\tbp\|_\Gamma \|\tbq\|_\Gamma.
\end{align}
Note also that \eqref{piola_div} yields ${\rm div}_\Gamma \tbp= \mu f^\ell$, and the {weak} commuting diagram property {\eqref{wcdp}} then gives 
\begin{align} 
\label{interp_piola_div}
{b_\Gamma( {\rm div}_\Gamma I_h \tbp, v) = (\mu f^\ell, \pi_h v)_\Gamma, ~v \in L_2(\Gamma).}
\end{align}
  In addition, \eqref{mu_def} yields $\int_\Gamma \mu f^\ell = \int_\gamma f =0$, so $f_h = \mu f^\ell$ is a natural choice of right hand side for the discrete problem.  However, this choice may not always be practical and so we do not assume it below.  Whatever our choice of $f_h \in L_2^0(\Gamma)$, these observations along with \eqref{mixed_fem} yield
\begin{align}
\label{b_div_ident}
b_\Gamma(I_h \tbp-\bp_h, v_h )= (\mu f^\ell-f_h, v_h), ~~v_h \in V_h.
\end{align}

We first prove a preliminary bound for $\|I_h \tbp- \bp_h\|_\Gamma$.  
\begin{lemma} {Let $(\bp, u)$ and $(\bp_h, u_h)$ solve the continuous and discrete mixed problems \eqref{mixed_prob} and \eqref{mixed_fem}, respectively.}  
Let $\ebp=I_h \tbp-\bp_h$, $e_u= \pi_h u^\ell -u_h$, and $\overline{e_u}= \frac{1}{|\Gamma|} \int_\Gamma e_u$.  Then 
\begin{align}
\label{up_error}
\|\ebp\|_\Gamma + \|e_u-\overline{e_u} \|_\Gamma \lesssim \|\tbp-I_h \tbp\|_\Gamma + \|\Pi-\B_h\|_{L_\infty(\Gamma)} \|\bp\|_\gamma + \|\mu f^\ell -f_h\|_{\Gamma}.
\end{align}
\end{lemma}
\begin{proof}
  We first compute that 
\begin{align}
\begin{aligned}
\|\ebp\|_\Gamma^2 & = a_\Gamma(I_h \tbp-\tbp, \ebp) + a_\Gamma(\tbp-\bp_h, \ebp) &  
\\ & =  a_\Gamma(I_h \tbp-\tbp, \ebp) + [a_\Gamma(\tbp, \ebp)-a (\bp, \widetilde{\ebp})] + a(\bp, \widetilde{\ebp}) -b_\Gamma(\ebp, u_h) &  \eqref{mixed_fem}
\\ & = a_\Gamma(I_h \tbp-\tbp, \ebp) + a([\B_h-\Pi] \bp, \widetilde{\ebp})+ b(\widetilde{\ebp}, u)-b_\Gamma(\ebp, u_h)& \eqref{mixed_prob}, \eqref{a_pert}
\\ & = a_\Gamma(I_h \tbp-\tbp, \ebp) + a([\B_h-\Pi] \bp, \widetilde{\ebp})+ b_\Gamma(\ebp, u^\ell)-b_\Gamma(\ebp, u_h) &  \eqref{b_equiv}
\\ & = a_\Gamma(I_h \tbp-\tbp, \ebp) + a([\B_h-\Pi] \bp, \widetilde{\ebp}) +{b_\Gamma(\ebp, e_u)} & {\eqref{wcdp},  I_h \ebp=\ebp}  
\\ & = a_\Gamma(I_h \tbp-\tbp, \ebp) + a([\B_h-\Pi] \bp, \widetilde{\ebp})+ (\mu f^\ell-f_h, e_u)& \eqref{b_div_ident}
\\ & = a_\Gamma(I_h \tbp-\tbp, \ebp) + a([\B_h-\Pi] \bp, \widetilde{\ebp})+ (\mu f^\ell-f_h, e_u-\overline{e_u}) &  \int_\Gamma \mu f^\ell=\int_\Gamma f_h=0
\\ & \le [C\|I_h \tbp-\tbp\|_\Gamma^2 + \frac{1}{4} \|\ebp\|_\Gamma^2] + [ C \|\Pi-\B_h\|_{L_\infty(\Gamma)}^2 \|\bp\|_{\gamma}^2 + \frac{1}{4} \|\ebp\|_\Gamma^2] & 
\\ & ~~~~+ \|\mu f^\ell - f_h\|_\Gamma \|e_u-\overline{e_u}\|_{\Gamma}. & 
\end{aligned}
\end{align}
Reabsorbing the terms of $\|\ebp\|_{\Gamma}^2$ back into the left hand side yields
\begin{align}
\label{intermediate_ep}
\|\ebp\|_{\Gamma}^2  \lesssim \|\tbp-I_h \tbp\|_{\Gamma}^2 + \|\Pi-\B_h\|_{L_\infty(\Gamma)}^2 \|\bp\|_\gamma^2 + \|\mu f^\ell -f_h\|_{\Gamma} \|e_u-\bar{e_u}\|_{\Gamma}.
\end{align}

We next bound $\|e_u-\overline{e_u}\|_{\Gamma}$.  By the inf-sup condition \eqref{discrete_inf_sup} there is $\bq_h \in \bQh$ with $\|\bq_h\|_{\hdivGam}=1$ such that
\begin{align}
\label{begin_eu}
\|e_u-\overline{e_u}\|_\Gamma \lesssim b_\Gamma(\bq_h, e_u-\overline{e_u}) = b_\Gamma(\bq_h, e_u).
\end{align}
Using {\eqref{wcdp}}, \eqref{b_equiv}, \eqref{mixed_fem},  \eqref{mixed_prob}, \eqref{a_pert}, and norm equivalence \eqref{norm_equiv} we then compute
\begin{align}
\begin{aligned}
b_\Gamma(\bq_h, e_u) & = b_\Gamma(\bq_h, u^\ell)-b_\Gamma(\bq_h, u_h)
\\ & = b(\widetilde{\bq_h}, u)-b_\Gamma (\bq_h, u_h)
\\ & = a(\bp,\tbpq)-a_\Gamma(\tbp, \bq_h) + a_\Gamma(\tbp-\bp_h, \bq_h)
\\ & \lesssim \|\Pi-\B_h\|_{L_\infty(\Gamma)} \|\bp\|_{\gamma} \|\tbpq\|_\gamma + \|\tbp-\bp_h\|_{\Gamma} \|\bq_h\|_\Gamma
\\ & \lesssim (\|\Pi-\B_h\|_{L_\infty(\Gamma)} \|\bp\|_\gamma + \|\tbp-I_h \tbp\|_\Gamma + \|\ebp\|_\Gamma) \|\bq_h\|_\Gamma.
\end{aligned}
\end{align}
Recalling that $\|\bq_h\|_{\hdivGam} =1$ and inserting into \eqref{begin_eu} yields
\begin{align}
\label{intermediate_eu}
\|e_u-\overline{e_u}\| \lesssim \|\Pi-\B_h\|_{L_\infty(\Gamma)} \|\bp\|_\gamma + \|\tbp-I_h \tbp\|_\Gamma + \|\ebp\|_\Gamma.
\end{align}

Next we insert \eqref{intermediate_eu} into \eqref{intermediate_ep} and apply Young's inequality to obtain
\begin{align}
\begin{aligned}
\|\ebp\|_\Gamma^2 & \le C \Big [ \|\tbp-I_h \tbp\|_\Gamma^2 + \|\Pi-\B_h\|_{L_\infty(\Gamma)}^2 \|\bp\|_\gamma^2 
\\ & ~~~~+ \|\mu f^\ell-f_h\|_\Gamma ( \|\Pi-\B_h\|_{L_\infty(\Gamma)} \|\bp\|_\gamma + \|\tbp-I_h \tbp\|_\Gamma + \|\ebp\|_\Gamma)
 \Big ] 
 \\ & \le C \left [ \|\tbp-I_h \tbp\|_\Gamma^2 + \|\Pi-\B_h\|_{L_\infty(\Gamma)}^2 \|\bp\|_\gamma^2 + \|\mu f^\ell - f_h\|_{\Gamma}^2 \right ] + \frac{1}{2} \|\ebp\|_\Gamma^2.
 \end{aligned}
\end{align}
Absorbing the last term into the left hand side and taking a square root completes the proof of the first estimate in \eqref{up_error}.  Inserting this result into \eqref{intermediate_eu} completes the proof of the bound for the second term.
\end{proof}

\subsection{Superconvergence in the scalar variable}
We next prove a classical superconvergence estimate for $e_u$ which yields a higher order of convergence than \eqref{up_error} when Raviart-Thomas elements are employed.
\begin{lemma}
\begin{align}
\label{super_eu}
\begin{aligned}
\|e_u\|_\Gamma & \lesssim  h \|\tbp-\bp_h\|_\Gamma +h {\|\mu_h \mu f^\ell -\pi_h (\mu_h \mu f^\ell)\|_{\Gamma}}  
\\ & ~~~~+  \|\Pi-\B_h\|_{L_\infty(\Gamma)} \|\bp\|_\gamma + \|\mu f^\ell - f_h\|_{\Gamma} + \|1-\mu\|_{L_\infty(\Gamma)} \|u\|_\gamma.  
\end{aligned} 
\end{align}
Thus for either of the element choices {$BDM_r$ or $RT_{r-1}$, $r \ge 1$, and with $f$ sufficiently smooth and $f_h$ chosen so that $\|\mu f^\ell -f_h\| \lesssim h^{k+1}\|f\|_\gamma$ } we have
\begin{align}
\label{super_eu2}
\|e_u\|_\Gamma \lesssim {h^{r+1} (\|u\|_{H^{r+1}(\gamma)} + \|f\|_{H_h^r(\gamma)})+h^{k+1}\|f\|_{\gamma} }.
\end{align}
\end{lemma}

\begin{proof}
Let $\phi \in H^2(\gamma) \cap L_2^0(\gamma)$ satisfy $-\Delta_\gamma \phi = \frac{1}{\mu} (e_u-\overline{e_u})^\ell$.  Then by standard $H^2$ regularity results, \eqref{mu_bound}, and norm equivalence \eqref{norm_equiv}, 
\begin{align}
\label{eu_reg}
\|\phi\|_{H^2(\gamma)} \lesssim \|\frac{1}{\mu}(e_u-\overline{e_u})^\ell\|_\gamma \lesssim \|e_u-\overline{e_u}\|_\Gamma.
\end{align}
Let also $\bz= {-}\nabla_\gamma \phi$.  Then ${\rm div}_\gamma \bz = \frac{1}{\mu}(e_u-\overline{e_u})^\ell$.  By the {weak} commuting diagram property {\eqref{wcdp}} and property \eqref{piola_div} of the Piola transform we than have that {${\rm div}_\Gamma \tbz=e_u-\overline{e_u}$ and subsequently $b_\Gamma(\tbz, \pi_h v) =b_\Gamma (I_h \tbz, v)$}.  Thus recalling the definition $e_u=\pi_h u^\ell -u_h= \pi_h (u^\ell -u_h)$, we compute
\begin{align}
\label{eu_1}
\|e_u-\overline{e_u}\|_\Gamma^2 {=b_\Gamma(\tbz, e_u-\overline{e_u})}= b_\Gamma(I_h \tbz, e_u-\overline{e_u}) = b_\Gamma (I_h \tbz, e_u) = b_\Gamma (I_h \tbz, u^\ell -u_h).
\end{align}
Using the definition of $u_h$ via the mixed FEM \eqref{mixed_fem} and the transformation identity \eqref{b_equiv}, we next compute after inserting $\pm a_\Gamma(\tbp, I_h \tbz)$ that
\begin{align}
\label{eu_2}
\begin{aligned}
b_\Gamma& (I_h \tbz, u^\ell) -b_\Gamma (I_h \tbz, u_h) = b(\widetilde{I_h \tbz}, u) - a_\Gamma (\bp_h, I_h \tbz)
\\ & = a(\bp, \widetilde{I_h \tbz}) - a_\Gamma(\tbp, I_h \tbz)+ a_\Gamma(\tbp, I_h \bz)-a_\Gamma(\bp_h, I_h \tbz)
\\ & \le \|\Pi-\B_h\|_{L_\infty(\Gamma)} \|\bp \|_{\gamma} \|I_h \tbz\|_\Gamma + |a_\Gamma(\tbp-\bp_h, I_h \tbz-\tbz)| + |a_\Gamma (\tbp-\bp_h, \tbz)|
\\ & \lesssim \|\Pi-\B_h\|_{L_\infty(\Gamma)} \|\bp \|_{\gamma} \|I_h \tbz\|_\Gamma +h \|\tbp-\bp_h\|_\Gamma \|\tbz\|_{H^1(\Gamma)} + |a_\Gamma (\tbp-\bp_h, \tbz)|.
\end{aligned}
\end{align}
{Using the definitions of $\phi$ and $z$, \eqref{b_equiv}, \eqref{piola_div}, \eqref{wcdp}, \eqref{mixed_fem}, the definitions of $\mu_h$ and $\pi_h^{lin}$, and norm equivalence,} we finally compute that  
\begin{align}
\label{eu_3}
\begin{aligned}
a_\Gamma(\tbp-\bp_h, \tbz) & = a_\Gamma(\tbp-\bp_h, \tbz)-a(\bp-\tbph, \bz) + a(\bp-\tbph, \bz)
\\ & = a_\Gamma(\tbp-\bp_h, \tbz)-a(\bp-\tbph, \bz) + b(\bp-\tbph, \phi)
\\ & = a_\Gamma(\tbp-\bp_h, \tbz)-a(\bp-\tbph, \bz) + b_\Gamma(\tbp-\bp_h, \phi^\ell)
\\ & =a_\Gamma(\tbp-\bp_h, \tbz)-a(\bp-\tbph, \bz)+ {(\mu f^\ell, \phi^\ell)_\Gamma-b_\Gamma(p_h, \pi_h \phi^\ell)} 
\\ & = a_\Gamma(\tbp-\bp_h, \tbz)-a(\bp-\tbph, \bz) + {(\mu f^\ell, \phi^\ell)_\Gamma - (f_h, \pi_h \phi^\ell)_\Gamma } 
\\ & = a_\Gamma(\tbp-\bp_h, \tbz)-a(\bp-\tbph, \bz)+ {(\mu_h \mu f^\ell, \phi^\ell)_{\Gamma^{lin}} - (\mu_h f_h, \pi_h^{lin} \phi^\ell)_{\Gamma^{lin}}}
\\ & = a_\Gamma(\tbp-\bp_h, \tbz)-a(\bp-\tbph, \bz) + {(\mu_h \mu f^\ell - \pi_h^{lin} (\mu_h \mu f^\ell), \phi^\ell)_{\Gamma^{lin}}}
\\ & ~~~~{ + (\pi_h^{lin} (\mu_h \mu f^\ell), \phi^\ell)_{\Gamma^{lin}}- (\mu_h f_h, \pi_h^{lin} \phi^\ell)_{\Gamma^{lin}}}
\\ & = a_\Gamma(\tbp-\bp_h, \tbz)-a(\bp-\tbph, \bz) + {(\mu_h \mu f^\ell - \pi_h (\mu_h \mu f^\ell), \phi^\ell-\pi_h \phi^\ell)_{\Gamma^{lin}} }
\\ & ~~~~{+ (\mu f^\ell-f_h, \pi_h \phi^\ell)_{\Gamma}}
\\ & \lesssim \|\Pi - \B_h\|_{L_\infty(\Gamma)} \|\tbp-\bp_h\|_\Gamma  \|\bz\|_\gamma 
\\ & ~~~~+ \|\phi\|_{H^1(\gamma)}  \left [h {\|\mu_h \mu f^\ell -\pi_h (\mu_h \mu f^\ell)\|_{\Gamma}}  + \|\mu f^\ell -f_h\|_\Gamma \right ].
\end{aligned}
\end{align}

Combining \eqref{eu_reg} through \eqref{eu_3} with \eqref{interp_bounds} then yields
\begin{align}
\begin{aligned}
\|e_u-\overline{e_u}\|_\Gamma^2 & \lesssim \|\phi\|_{H^2(\gamma)} \big [ \|\Pi-\B_h\|_{L_\infty(\Gamma)} \|\bp\|_\gamma + h \|\tbp-\bp_h\|_\Gamma 
\\ & ~~~~+ h \|\mu_h {\mu f^\ell -\pi_h (\mu_h \mu f^\ell)\|_{\Gamma}} + \|\mu f^\ell -f_h\|_\Gamma \big ] 
\\ & \lesssim \|e_u-\overline{e_u}\|_\Gamma \big [ \|\Pi-\B_h\|_{L_\infty(\Gamma)} \|\bp\|_\gamma + h \|\tbp-\bp_h\|_\Gamma 
\\ & ~~~~+ h {\|\mu_h \mu f^\ell -\pi_h (\mu_h \mu f^\ell)\|_{\Gamma}} + \|\mu f^\ell -f_h\|_\Gamma \big ].
\end{aligned}
\end{align}
Dividing through by $\|e_u-\overline{e_u}\|_\Gamma$ completes the proof of \eqref{super_eu} up to a term $\overline{e_u}$ remaining on the left hand side.  To complete the proof we recall that $0=\int_\Gamma u_h = \int_\gamma u = \int_\Gamma \mu u^\ell$.  Thus $|\Gamma| \overline{e_u} = \int_\Gamma (\pi_h u^\ell-u_h) = \int_\Gamma u^\ell = \int_\Gamma (1-\mu) u^\ell \le \|1-\mu\|_{L_\infty(\Gamma)} \|u\|_\gamma$, which completes the proof of {\eqref{super_eu}}  after application of the triangle inequality.  {The result \eqref{super_eu2} follows after application of approximation and geometric error estimates and noting that $\|u\|_{H^1(\gamma)} \lesssim \|f\|_{\gamma}$.} 
\end{proof}

 \subsection{Summary of convergence results}
Collecting the previous error estimates while recalling the approximation results \eqref{interp_bounds}  and recalling the geometric error bound \eqref{B_bound} easily yields the following.
\begin{theorem}[Summary of convergence results]
\label{theorem1}
Assume that $f_h$ is defined so that $\|\mu f^\ell -f_h\| \lesssim {h^{k+1}} \|f\|_\gamma$, {$f \in H_h^{r+1}(\gamma)$}, and {$u \in  H^{r+2}(\gamma)$.   Then for $RT_r$ elements, $r \ge 0$,} we have
\begin{align}
\|e_u\|_{\Gamma} + h (\|\tbp-\bp_h\|_{\Gamma} + \|u^\ell -u_h\|_{\Gamma}) \lesssim { h^{r+2} (\|u\|_{H^{r+2}(\gamma)}+ \|f\|_{H_h^{r+1}(\gamma)})  + \G(h^{k+1})} \|f\|_\gamma,
\end{align}
where {$\G(h^{k+1}) \lesssim h^{k+1}$} is a geometric error term arising from the approximation of $\gamma$ by $\Gamma$.  In the case of {$BDM_r$} elements we instead have
\begin{align}
\|e_u\|_\Gamma + \|\tbp-\bp_h\|_\Gamma + h \|u^\ell -u_h\|_\Gamma \lesssim {h^{r+1} ( \|u\|_{H^{r+2}(\gamma)}+ \|f\|_{H_h^r(\gamma)}) + \G(h^{k+1})} \|f\|_\gamma.
\end{align}
\end{theorem}
 
\subsection{Superconvergent postprocessing}

{In this subsection we prove error estimates for the postprocessing methods \eqref{post_def} and \eqref{post_def2} which provide potentially higher-order approximations to the scalar variable.   Recall that in these definitions we employ either $RT_{r-1}$ or $BDM_r$ elements in our finite element method and recover a postprocessed approximation to $u$ in locally $\mathbb{P}_r$ spaces.  Our} error estimates also include the effects of geometric errors. To our knowledge such proofs are lacking in the literature for surface FEM even assuming a shape regular triangulation.  The anisotropic meshes that we consider here also introduce some modest additional technical complications.

\begin{theorem}[Superconvergent postprocessing]
\label{theorem2}
Let $I_L$ be the standard Lagrange interpolant into the piecewise {$\mathbb{P}_r$} elements on $\F_h$.  For both of the above postprocessing techniques and {both choices $BDM_r$ and $RT_r$ of elements defined above}, we have
\begin{align}
\label{post_result}
\begin{aligned}
\|u^\ell -u_h^*\|_\Gamma & \lesssim \big [ \|u^\ell -I_L u^\ell\|_\Gamma +  h\|\nabla_\Gamma (u^\ell -I_L u^\ell)\|_\Gamma +
\\ & ~~~~~~~ {h \|u^\ell -\pi_h u^\ell\|_\Gamma} 
+  h \|\tbp - I_h \tbp\|_\Gamma + h \|{\mu_h}\mu f^\ell - \pi_h ({\mu_h} \mu f^\ell)\|_\Gamma  \big ] 
\\ & ~~~~~+ \left [ \|\Pi-\B_h\|_{L_\infty(\Gamma)} \|\bp\|_\gamma + \|1-\mu\|_{L_\infty(\Gamma)} \|u\|_\gamma + \|\mu f^\ell - f_h \|_\Gamma \right ]
\\ & ~~~~+ \left [ h\|\nabla_\Gamma u^\ell + \tbp\|_\Gamma  + h^2 \|f_h^* -{\mu_h^{-1} \pi_h (\mu_h f_h)}\|_\Gamma  \right ] 
\\ & := I+II+III.
\end{aligned}
\end{align}
Thus if ${\mu_h} \mu f^\ell$ is elementwise in {$H^{r}$} or equivalently $f$ is elementwise in {$H^r$}, $f_h$ approximates $\mu f^\ell$ to order ${h^{k+1}}$, {$\|f_h^* -{\mu_h^{-1} \pi_h (\mu_h f_h)}\|_\Gamma \lesssim h^{r-1}$},  and we employ $RT_{r-1}$ or $BDM_r$ elements, we have
\begin{align}
\label{postproc_order}
\|u^\ell - u_h^*\|_T \lesssim {  h^{r+1} ( \|u\|_{H^{r+1}(\gamma)}+ \|f\|_{H_h^r(\gamma)}) + h^{k+1}\|f\|_{\gamma}} .  
\end{align}
\end{theorem}

\begin{remark}
\label{rem0}
{The postprocessing defined by \eqref{post_def2} does not involve $f_h^*$, so the term $h^2 \|f_h^* -{\mu_h^{-1} \pi_h (\mu_h f_h)}\|_\Gamma$ may be omitted from \eqref{post_result} when this method is used.  When \eqref{post_def} is employed, it is not difficult or expensive to compute $f_h^* =\mu_h^{-1} \pi_h (\mu_h f_h)$ since $f_h$ and $\mu_h$ are already needed in the code and $\pi_h$ is defined elementwise.  With this definition errors due to the definition of $f_h^*$ can be eliminated.  On the other hand,  it may be more convenient to use the second method \eqref{post_def2} as doing so completely eliminates the need to define and compute $f_h^*$.}  
\end{remark}

\begin{remark}  
\label{rem1}
We remark on a subtlety in the structure of the geometric errors in the postprocessing estimate \eqref{post_result}.  Term $I$ consists of approximation terms which would essentially also be present in the Euclidean case.  Term $II$ consists of geometric error terms of order {$h^{k+1}$} that arise from the definition of the original mixed solution $(\bp_h, u_h)$, but these terms {\it do not arise from the definition of the postprocessed solution}.  Term $III$ consists of data approximation and geometric error terms arising from the postprocessing technique.  Note that by \eqref{piola_dif} $III={O(h^{r+2}+h^{k+2})}$ under reasonable assumptions on $f_h^*$, {which for any choice of $r \ge 1$ translates to $O(h^{3})$ on the simplest case $k=1$ of affine surface approximations.}   That is, the geometric errors present in the original solution are present with the same order in the postprocessed solution, but the geometric errors arising from the postprocessing technique itself are actually of higher order.  Consider for example the case of a parametric finite element solution using $BDM_2$ elements and with the finite element method defined on the exact surface $\gamma$ instead of on the approximation $\Gamma$; the error form above is then still valid.  No geometric errors are then present in the definition of $(\bp_h, u_h)$, so we expect $O(h^3)$ convergence in the vector variable and in $e_u$ and $O(h^2)$ in $u-u_h$; cf. \cite{CD16}.  In this case $II$ disappears from \eqref{post_result}, indicating that the postprocessed solution will converge with order $h^3$ {\it even if the postprocessing itself is defined on the affine surface approximation $\Gamma$}.  We illustrate this phenomenon in our numerical experiments below.  
\end{remark}

\begin{proof}
Let $I_{L}$ be the standard Lagrange interpolant onto {$\mathbb{P}_r(T)$}, and let $e_T={I_L} u^\ell - u_h^*$.  We first carry out an auxiliary calculation in order to bound the mean value of $I_L u^\ell - u_h^*$ on $T$.  Let $e_T= I_L u^\ell - u_h^*$ and $\overline{e_T}= \frac{1}{|T|} \int_T e_T$.  {Note first that given $T \in \F_h$ with $\F_h^{lin} \ni \overline{T} = \Theta^{-1}(T)$, we have $\int_T (u^\ell -\pi_h u^\ell ) = \int_{\overline{T}} (u-\pi_h^{lin} u^\ell ) \mu_h= \int_{\overline{T}} (u^\ell -\pi_h^{lin} u^\ell ) (\mu_h-1)= \int_T \frac{1}{\mu_h} ( u^\ell  -\pi_h u^\ell ) (\mu_h-1)$. Using this observation}, \eqref{super_eu2}, {and the second line of \eqref{post_def}} we compute
\begin{align}
\label{mean_bound}
\begin{aligned}
\|\overline{e_T} \|_T  &\le \frac{1}{|T|} |T|^{1/2} \left |\int_T (I_L u^\ell - u^\ell + u^\ell {-\pi_h u^\ell + \pi_h u^\ell} - u_h^*) \right |
\\ & \lesssim |T|^{-1/2}  (\|I_L u^\ell - u^\ell\|_{L_1(T)} +{\|1-\mu_h\|_{L_\infty(T)} \|u^\ell -\Pi_h u^\ell \|_{L_1(T)})}+ \left | \int_T ( \pi_h u^\ell -u_h^*) \right | 
\\ &  \le  \|I_L u^\ell - u^\ell\|_T + \|1-\mu_h\|_{L_\infty(T)} \|u^\ell -\pi_h u^\ell \|_T + |T|^{-1/2} \left |\int_T \pi_h u^\ell -u_h \right | 
\\ & \le \|I_L u^\ell -u^\ell \|_T + \|1-\mu_h \|_{L_\infty(T)} \|u^\ell -\pi_h u^\ell \|_T + |T|^{-1/2} |T|^{1/2} \|e_u\|_T
\\ & = \|I_L u^\ell -u^\ell\|_T {+\|1-\mu_h\|_{L_\infty(T)}\|u^\ell -\pi_h u^\ell \|_T }+  \|e_u \|_T.
\end{aligned}
\end{align}
Employing a Poincar\'e inequality, \eqref{mean_bound}, {and \eqref{discrete_surfs},} we then have
\begin{align}
\label{eq498}
\begin{aligned}
\|u^\ell-u_h^*\|_T & \le \|u^\ell -I_L u^\ell \|_T + \|e_T-\overline{e_T}\|_T + \|\overline{e_T}\|_T
\\ & \lesssim \|u^\ell - I_L u^\ell \|_T + \|e_u\|_T  {+h \|u^\ell -\pi_h u^\ell \|_T } + h_T \|\nabla_\Gamma e_T\|_T.
\end{aligned} 
\end{align}

Next we compute that
\begin{align}
\label{eq499}
\begin{aligned}
\| \nabla e_T\|_T^2 & = \int_T \nabla_\Gamma e_T \cdot \nabla_\Gamma e_T
\\ & = \int_T \nabla_\Gamma (I_L u^\ell - u^\ell) \cdot \nabla_\Gamma e_T + \int_T \nabla_\Gamma (u^\ell - u_h^*) \cdot \nabla_\Gamma e_T
\\ & \le \|\nabla_\Gamma (u^\ell - I_L u^\ell)\|_T \|\nabla e_T\|_T + \left | \int_T \nabla_\Gamma (u^\ell - u_h^*) \cdot \nabla_\Gamma e_T \right|.
\end{aligned}
\end{align}
Note that $\nabla_\Gamma e_T  =\nabla_\Gamma (e_T-\overline{e_T})$, and in addition $e_T-\overline{e_T} \in {\mathbb{P}_r^0}$.

Considering now the first postprocessing technique \eqref{post_def}, we employ the first line of that definition to obtain
\begin{align}
\label{eq500}
\begin{aligned}
\int_T \nabla_\Gamma  (u^\ell - u_h^*) \cdot \nabla_\Gamma e_T & = \int_T \nabla_\Gamma u^\ell \cdot \nabla_\Gamma e_T - \int_T f_h^* (e_T-\overline{e_T})
\\& ~~~~+ \int_{\partial T} \bp_h \cdot {\bf n }_T (e_T - \overline{e_T}).
\end{aligned}
\end{align}
We first write
\begin{align}
\label{eq501}
\begin{aligned} 
\int_T \nabla_\Gamma u^\ell \cdot \nabla_\Gamma e_T & = \int_T [\nabla_\Gamma u^\ell+\tbp] \cdot \nabla_\Gamma e_T - \int_T \tbp \cdot \nabla_\Gamma e_T.
\end{aligned}
\end{align}
Also, applying the divergence theorem yields
\begin{align}
\label{eq502}
\begin{aligned} 
\int_{\partial T} \bp_h \cdot {\bf n }_T (e_T - \overline{e_T}) 
 = \int_T ( e_T- \overline{e_T}){\rm div}_\Gamma \bp_h + \int_T \bp_h \cdot \nabla_\Gamma e_T.
\end{aligned} 
\end{align}
{ We next note that $\piola_{\Theta_h^{-1}} \bp_h \in \bQ_h^{lin}$, so ${\rm div}_{\Gamma^{lin}} \piola_{\Theta_h^{-1}} \bp_h \in V_h^{lin}$.  Using \eqref{mixed_fem}, \eqref{b_equiv}, and the definition of $\mu_h$ then yields for $v_h \in V_h$ that $b_{\Gamma^{lin}} (\piola_{\Theta_h^{-1}} \bp_h, v_h) = b_{\Gamma} (\bp_h, v_h) = (f_h, v_h)_{\Gamma} = (\mu_h f_h, v_h)_{\Gamma^{lin}}$.  Thus ${\rm div}_{\Gamma^{lin}} \piola_{\Theta_h^{-1}} \bp_h = \pi_h^{lin} (\mu_h f_h)$, and by \eqref{piola_div} we finally conclude that ${\rm div}_\Gamma \bp_h = \mu_h^{-1} \pi_h (\mu_h f_h)$.  Using this fact, } collecting \eqref{eq501} and \eqref{eq502} into \eqref{eq500}, and finally employing a Poincar\'e inequality then yields
\begin{align}
\label{eq503}
\begin{aligned}
\int_T & \nabla_\Gamma(u^\ell  - u_h^*)\cdot \nabla_\Gamma e_T =\int_T [ \nabla_\Gamma u^\ell + \tbp] \cdot \nabla_\Gamma e_T
\\ & ~~~~+  \int_T ({\mu_h^{-1}  \pi_h (\mu_h f_h)} -f_h^*) (e_T-\overline{e_T})  - \int_T (\tbp-\bp_h) \cdot \nabla_\Gamma e_T
\\ & \le \|\nabla_\Gamma u^\ell + \tbp\|_T \|\nabla_\Gamma e_T\|_T + \|f_h^* -{\mu_h^{-1} \pi_h(\mu_h f_h)}\|_T \|e_T-\overline{e_T}\|_T 
\\ & ~~~~+ \|\tbp-\bp_h\|_T \|\nabla_\Gamma e_T\|_T
\\ & \lesssim  (\|\nabla_\Gamma u^\ell + \tbp\|_T  + h_T \|f_h^* -{\mu_h^{-1} \pi_h (\mu_h f_h)}\|_T  + \|\tbp-\bp_h\|_T) \|\nabla_\Gamma e_T\|_T.
\end{aligned}
\end{align}
We now insert \eqref{eq503} into \eqref{eq499} and then cancel a factor of $\|\nabla_\Gamma e_T\|_T$ to obtain
\begin{align}
\label{eq504}
\begin{aligned}
\|\nabla_\Gamma e_T\|_T & \lesssim \|\nabla_\Gamma (u^\ell -I_h u^\ell)\|_T + \|\nabla_\Gamma u^\ell + \tbp\|_T 
\\ & ~~~~+ h_T \|f_h^* -{\mu_h^{-1} \pi_h (\mu_h f_h)}\|_T + \|\tbp-\bp_h\|_T.
\end{aligned}
\end{align}
Finally inserting \eqref{eq504} into \eqref{eq498} then yields
\begin{align}
\label{eq506}
\begin{aligned}
\|u^\ell -u_h^*\|_T & \lesssim \|u^\ell -I_L u^\ell\|_T + \|e_u\|_T + h_T\|\nabla_\Gamma (u^\ell -I_L u^\ell)\|_T 
\\ & ~~~~+ h_T\|\nabla_\Gamma u^\ell + \tbp\|_T  + h_T^2 \|f_h^* -{\mu_h^{-1} \pi_h (\mu_h f_h)}\|_T + h_T \|\tbp-\bp_h\|_T.
\end{aligned}
\end{align}

Assume now that we have used \eqref{post_def2} instead of \eqref{post_def} in order to define $u_h^*$.  In place of \eqref{eq500} we then have
\begin{align}
\label{eq505}
\begin{aligned}
\int_T \nabla_\Gamma (u^\ell -u_h^*) \cdot \nabla_\Gamma e_T & = \int_T ( \nabla_\Gamma u^\ell + \bp_h) \cdot \nabla_\Gamma e_T
\\ & = \int_T (\nabla_\Gamma u^\ell + \tbp -\tbp + \bp_h) \cdot \nabla_\Gamma e_T
\\ & \le (\|\nabla_\Gamma u^\ell + \tbp\|_T + \|\tbp-\bp_h\|_T) \|\nabla_\Gamma e_T\|_T.
\end{aligned}
\end{align}
Inserting \eqref{eq505} into \eqref{eq499}, dividing through by $\|\nabla_\Gamma e_T\|_T$, and inserting the result into \eqref{eq498} then yields
\begin{align}
\label{eq507}
\begin{aligned}
\|u^\ell -u_h^*\|_T & \lesssim \|u^\ell-I_L u^\ell \|_T + \|e_u\|_T +  h_T\|\nabla_\Gamma (u^\ell -I_L u^\ell)\|_T
\\ & ~~~~+ h_T ( \|\nabla_\Gamma u^\ell + \tbp\|_T + \|\tbp-\bp_h\|_T).
\end{aligned}
\end{align}
The bound \eqref{post_result} may be obtained by inserting \eqref{super_eu} and \eqref{up_error} into \eqref{eq506} or \eqref{eq507} as appropriate, while \eqref{postproc_order} follows from \eqref{mu_bound}, \eqref{B_bound}, \eqref{piola_dif}, and the estimates in Subsection \ref{subsec:interp}.  
\end{proof}

\section{Numerical results}
\label{sec5}

\subsection{Description of setup}
In this section we provide numerical experiments to illustrate our theoretical results.  

In our numerical experiments we took $\gamma$ to be the unit sphere and our test solution to be $u(x,y,z)=sin(x)+y+z^3$.  We tested our method with the lowest-order finite element spaces $RT_0$ and $BDM_1$.  We also carried out experiments using the hybridizable discontinuous Galerkin space $HDG_1$ (see more discussion below).  The postprocessing technique \eqref{post_def} was used in all experiments, with computations carried out on the discrete surface $\Gamma$.  

In our numerical experiments practical solution of the finite element system \eqref{mixed_fem} is carried out by a hybridization procedure.   Such techniques are standard for mixed FEM and lead to a positive semidefinite system (or positive definite once the mean value condition $\int_\Gamma u_h=0$ is accounted for).   See for example Chapter 7 of \cite{BBF13}; hybridization for surface FEM is also discussed in the context of hybridizable DG methods in \cite{CD16}.  Roughly speaking, hybridizable methods eliminate elementwise defined degrees of freedom by condensing global system information into edgewise defined piecewise polynomial spaces of multipliers.  The global system matrix is defined with respect to these edge degrees of freedom, and is positive semidefinite (or positive definite once the mean value zero condition is accounted for).  The solutions $u_h$ and $\bp_h$ are recovered by elementwise operations after the global system is solved.  The edge multipliers also serve as approximations of the solution $u$.  

Conditioning and preconditioning of hybridized mixed systems in the Euclidean context is discussed in \cite{Gop03}, where it is shown that the condition number is roughly equivalent to that of a corresponding system matrix for a standard formulation.  In \cite{ORX12} the properties of the stiffness matrix for a standard Lagrange quasi-trace method for the Laplace-Beltrami problem are studied.  The smallest eigenvalue of the stiffness matrix is zero due to the mean value zero condition, so the effective condition number (ratio of the largest to the second eigenvalue) is considered and observed to be arbitrarily poor depending on the placement of the nodes of the bulk mesh $\T_h$ relative to the surface $\gamma$.  Given the similarity between the standard stiffness matrix and the hybrid mixed system matrix discussed in \cite{Gop03}, we expect a similar degeneration in our system matrix, and that is indeed what is observed in practice.  Thus preconditioned iterative methods are not straightforward to implement, and we instead used direct solvers (MATLAB ``backslash'').   Such solvers are adequate for problems of moderate size, but as pointed out in \cite{ORX12} it would be desirable to find a more efficient solver technology.   We focus mainly on error behavior and do not consider solver issues further here, but they remain an important outstanding issue in the practical implementation of quasi-trace methods.   We do note that the behavior of errors in our methods was stable with respect to node placement and effective condition number.  

\subsection{Experiment 1}
In our first experiment we consider the error behavior of our method using $RT_0$ and $BDM_1$ elements; see Figure \ref{fig3} below.  As expected from Theorem \ref{theorem1}, the lowest-order Raviart Thomas elements yield order of convergence $h$ for $\|u^\ell -u_h\|_\gamma$ and $\|\tbp-\bp_h\|_\gamma$, and we also obtain order $h^2$ convergence for $\|u^\ell -u_h^*\|_\gamma$ as expected from Theorem \ref{theorem2}.  We obtain the same orders of convergence for the lowest order $BDM_1$ elements, except for in the vector variable $\|\tbp-\bp_h\|_\gamma$ which converges with order $h^2$ (also as predicted).

\setlength{\unitlength}{.75cm}
\begin{figure}[h]
\centering
\includegraphics[scale=.4]{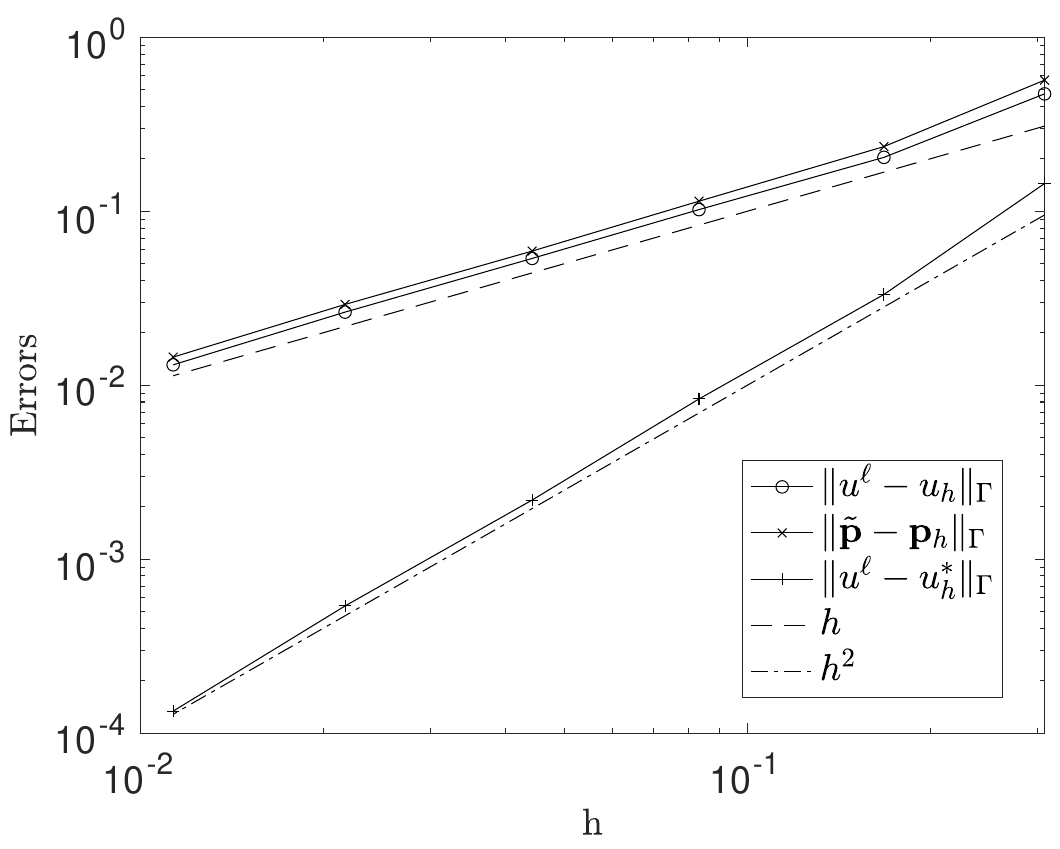}
\includegraphics[scale=.4]{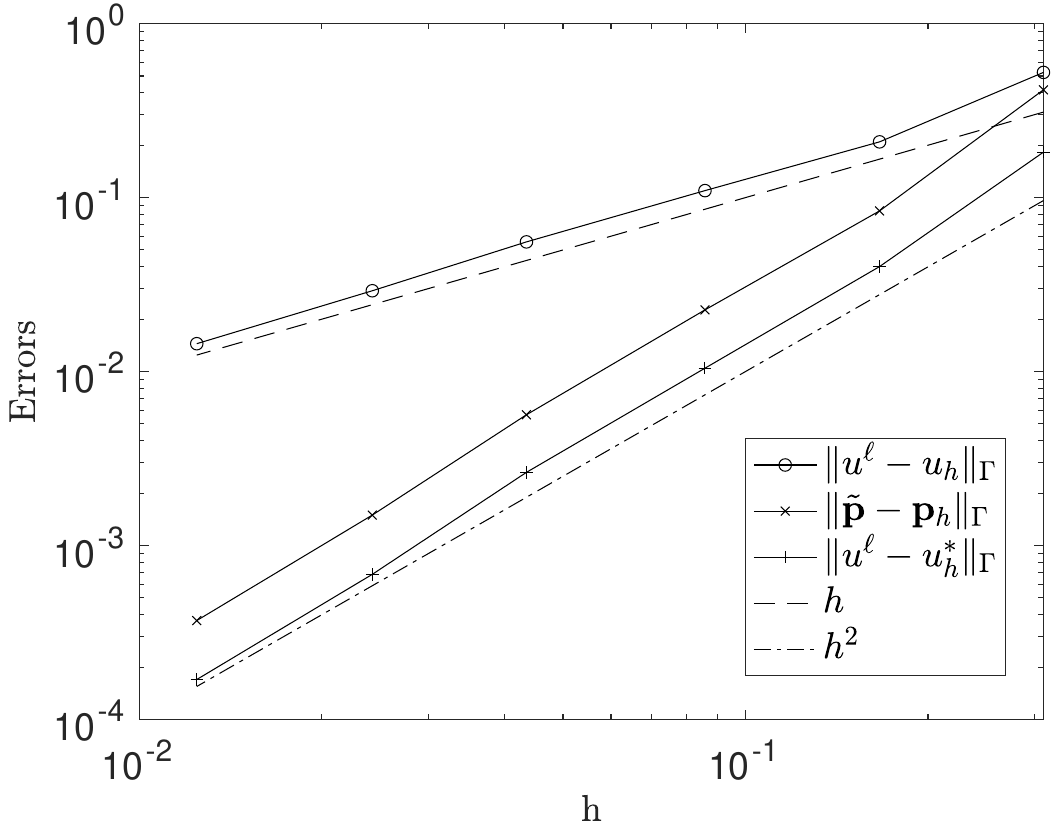}
\caption{Error behavior for $RT_0$ (left) and $BDM_1$ elements (right).}
\label{fig3}
\end{figure}

\subsection{Experiment 2}
In our second experiment we illustrate the use of hybridizable discontinuous Galerkin (HDG) methods and also illustrate the observation made in Remark \ref{rem1} that the geometric errors induced in the postprocessing procedures \eqref{post_def} and \eqref{post_def2} are of order $h^3$.  Surface HDG methods generalize surface mixed methods in their hybridized form and are described and analyzed in \cite{CD16}.  HDG spaces consist of the full vector polynomials of a given degree to approximate the vector variable $\bp$, the full scalar polynomials of the same degree to approximate the scalar variable $u$, and an interelement multiplier of the same degree as in the hybridized form of mixed methods.  Below we consider the space $HDG_1$, where all three of these spaces consist of affine polynomials.  The analysis of $HDG$ methods in \cite{CD16} includes optimal-order error estimates for $\|\tbp-\bp_h\|_\Gamma$ and $\|u^\ell-u_h\|_\gamma$ and a superconvergent estimate for $\|\pi_h u^\ell -u_h\|_\Gamma$ as above, but under the assumption of quasi-uniform meshes.  These error estimates could be extended to quasi-trace methods by proving suitable approximation properties for a certain HDG projection operator on anisotropic elements; this is possible but beyond the scope of this paper.   Assuming these estimates, the postprocessing analysis of Theorem \ref{theorem2} can be directly applied to HDG methods.  

We carried out two sets of computations using the $\mathbb{P}_1$ HDG space $HDG_1$.  In the first we defined all forms (integrals) needed in the method on the discrete surface $\Gamma$ as above.  Based on the analysis of \cite{CD16} and the postprocessing error estimates above, this leads us to expect an order of convergence of $h^2$ in $\|\tbp-\bp_h\|_\Gamma$, $\|u^\ell -u_h\|_\Gamma$, and $\|u^\ell -u_h^*\|_\Gamma$.  These orders of convergence are observed in Figure \ref{fig4}.  In our second set of computations we defined a {\it parametric} $HDG_1$ method in which all forms were defined on lifts of elements in $\F_h$ to the continuous surface $\gamma$; see \cite{CD16} for more explanation.  This eliminates geometric errors in the definition of $u_h$ and $\bp_h$ (some quadrature errors arise because integrals are computed over curved surfaces, but these are easy to control by using high-order quadrature rules).  The postprocessing procedure \eqref{post_def} was however still defined on the {\it discrete} surface $\Gamma$.  This procedure in effect eliminates the $O(h^2)$ geometric error terms $II$ from the postprocessing error estimate \eqref{post_result} above, leaving only the terms $I$ (approximation errors) and $III$ (geometric and data approximation errors).  The result is an $O(h^3)$ postprocessing error as noted in Remark \ref{rem1}, which we observe in Figure \ref{fig4} below.  

\setlength{\unitlength}{.75cm}
\begin{figure}[h]
\centering
\includegraphics[scale=.5]{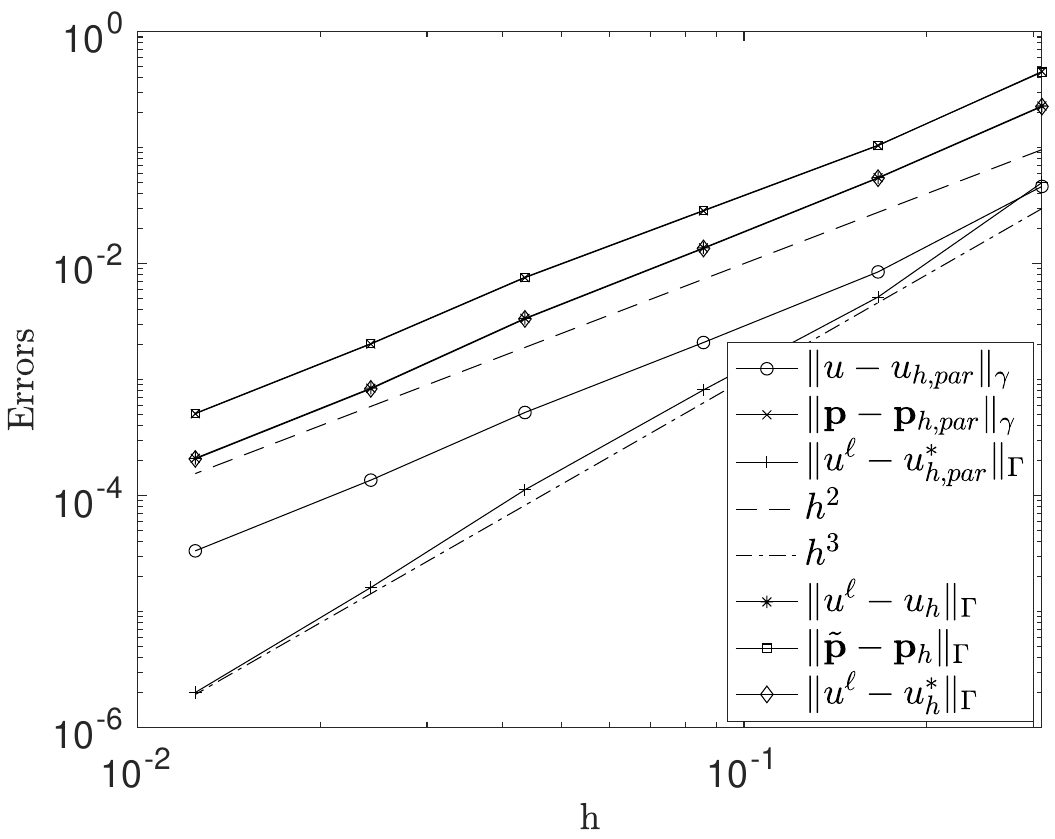}
\caption{Error behavior for $HDG_1$ elements using a parametric method ($\bp_{h, par}$, $u_{h,par}$, $u_{h,par}^*$) and the method defined on $\Gamma$ ($\bp_h$, $u_h$, $u_h^*$).}
\label{fig4}
\end{figure}

%
%

\begin{funding}
 This work was partially supported by NSF grant DMS-2012136. 
\end{funding}

\bibliographystyle{siam}
\bibliography{quasi_mixed}

\end{document}